\documentclass[reqno]{amsart}

\usepackage[utf8]{inputenc}
\usepackage{amssymb, amsmath, amsthm, color, enumerate, mathabx, mathrsfs}
\usepackage{bbm}

\usepackage{pdfpages}

\usepackage{hyperref}

\usepackage{tcolorbox}

\numberwithin{equation}{section}

\newcommand{\R}{\mathbb{R}}

\newcommand{\Z}{\mathbb{Z}}
\newcommand{\N}{\mathbb{N}}

\newcommand{\Sc}{\mathcal{S}}

\newcommand{\cL}{\mathcal{L}}

\newcommand{\cT}{\mathcal{T}}

\newcommand{\ud}{\mathrm{d}}

\newcommand{\bC}{\mathbf{C}}

\newcommand{\bJ}{\mathbf{J}}

\newcommand{\fa}{\mathfrak{a}}

\newcommand{\fJ}{\mathfrak{J}}

\newcommand{\fG}{\mathfrak{G}}

\def\inn#1#2{\langle#1,#2\rangle}

\theoremstyle{plain}

\newtheorem{theorem}{Theorem}[section]
\newtheorem{lemma}[theorem]{Lemma}
\newtheorem{proposition}[theorem]{Proposition}

\theoremstyle{definition}
\newtheorem{definition}[theorem]{Definition}

\newtheorem{remark}[theorem]{Remark}

\newcommand{\floor}[1]{\lfloor #1 \rfloor }
\newcommand{\ceil}[1]{\lceil #1 \rceil }
\newcommand{\be}{\mathbf{e}}
\newcommand{\xip}{\tfrac{\xi}{|\xi|}}

\newcommand{\supp}{\mathrm{supp}\,}
\newcommand{\xisupp}{\mathrm{supp}_{\xi}\,}

\newcommand{\dist}{\mathrm{dist}}

\def\bbone{{\mathbbm 1}}

\begin{document}





\title[Estimates for the helical maximal function]{Off-diagonal estimates for the helical maximal function}

\date{}


\begin{abstract}
The optimal $L^p \to L^q$ mapping properties for the (local) helical maximal function are obtained, except for endpoints. The proof relies on tools from multilinear harmonic analysis and, in particular, a localised version of the Bennett--Carbery--Tao restriction theorem. 
\end{abstract}

\author[D. Beltran]{ David Beltran }
\address{David Beltran: Departament d’An\`alisi Matem\`atica, Universitat de Val\`encia, Dr. Moliner 50, 46100 Burjassot, Spain}
\email{david.beltran@uv.es}
\thanks{D.B. supported by the AEI grant RYC2020-029151-I}

\author[J. Duncan]{ Jennifer Duncan}
\address{Jennifer Duncan:
ICMAT, C. de Nicolas Cabrera 13-15, 28049 Madrid, Spain}
\email{jennifer.duncan@icmat.es}
\thanks{J.D. partially supported by ERC grant 834728 and Severo Ochoa grant CEX2019-000904-S}

\author[J. Hickman]{ Jonathan Hickman }
\address{Jonathan Hickman: School of Mathematics, James Clerk Maxwell Building, The King's Buildings, Peter Guthrie Tait Road, Edinburgh, EH9 3FD, UK.}
\email{jonathan.hickman@ed.ac.uk}




\maketitle




\section{Introduction}




\subsection{Main results}

 Let $\gamma \colon I \to \R^3$ be a smooth curve, where $I := [-1,1]$, which is \textit{non-degenerate} in the sense that there is a constant $c_0 > 0$ such that 
\begin{equation}\label{eq:nondegenerate}
    |\det(\gamma^{(1)}(s), \gamma^{(2)}(s), \gamma^{(3)}(s))| \geq c_0 \qquad \textrm{for all $s \in I$.}
\end{equation}
This is equivalent to saying that $\gamma$ has non-vanishing curvature and torsion. Prototypical examples are the helix $\gamma(s)=(\cos(2\pi s), \sin(2\pi s), s)$ or the moment curve $\gamma(s)=(s,s^2/2,s^3/6)$. Given $t>0$, consider the averaging operator
\begin{equation*}
    A_tf(x) := \int_{\R} f(x - t\gamma(s))\,\chi(s)\,\ud s,
\end{equation*}
defined initially for Schwartz functions $f \in \Sc(\R^3)$, where $\chi \in C^{\infty}(\R)$ is a bump function supported on the interior of $I$. Furthermore, define the associated local maximal function
\begin{equation*}
    M_{\gamma} f(x):= \sup_{1 \leq t \leq 2} |A_t f(x)|.
\end{equation*}
Here we are interested in determining the sharp range of $L^p \to L^q$ estimates for $M_{\gamma}$. To describe the results, let 
\begin{equation*}
    \cT := \overline{\textrm{conv}\{(0,0), (1/3, 1/3), (1/4, 1/6)\}} \,\setminus\, \{(1/3, 1/3)\},
\end{equation*} 
so that $\cT$ is a closed triangle (formed by the closed convex hull of three points) with one vertex removed. We let $\mathrm{int}(\cT)$ denote the interior of $\cT$ and $\cL$ denote the intersection of $\cT$ with the diagonal: see Figure~\ref{fig: Riesz}. Standard examples show that $M_{\gamma}$ fails to be $L^p \to L^q$ bounded whenever $(1/p, 1/q) \notin \cT$: see \S\ref{sec:nec}. The following theorem therefore characterises the type set of $M_{\gamma}$, up to endpoints. 

\begin{theorem}\label{thm:main} For all $(1/p, 1/q) \in \mathrm{int}(\cT) \cup \cL$, there exists a constant $C_{\gamma, p,q} \geq 1$ such that the a priori estimate
\begin{equation*}
    \|M_{\gamma}f\|_{L^q(\R^3)} \leq C_{\gamma, p,q} \|f\|_{L^p(\R^3)}
\end{equation*}
holds for all $f \in \Sc(\R^3)$. 
\end{theorem}

For the diagonal case (that is, $(1/p, 1/q) \in \cL$), the sharp range of estimates was established in~\cite{BGHS-helical} and~\cite{KLO2022}, building on earlier work of~\cite{PS2007}. Hence, our main result is to push the range of boundedness to the region $\mathrm{int}(\cT)$.  As a consequence of Theorem \ref{thm:main} (or, more precisely, Theorem \ref{thm: Lp Lq loc smoothing} below) and~\cite[Theorem 1.4]{BRS}, $(p,q')$-sparse bounds for the full maximal operator $M_\gamma^{\mathrm{full}}f(x):=\sup_{t>0} |A_t f(x)|$  follow for $(1/p,1/q) \in \mathrm{int}(\cT)\cup \cL$. We omit the details and refer to~\cite{BRS} for the precise statements. 

\begin{figure}
    \centering
    \includegraphics{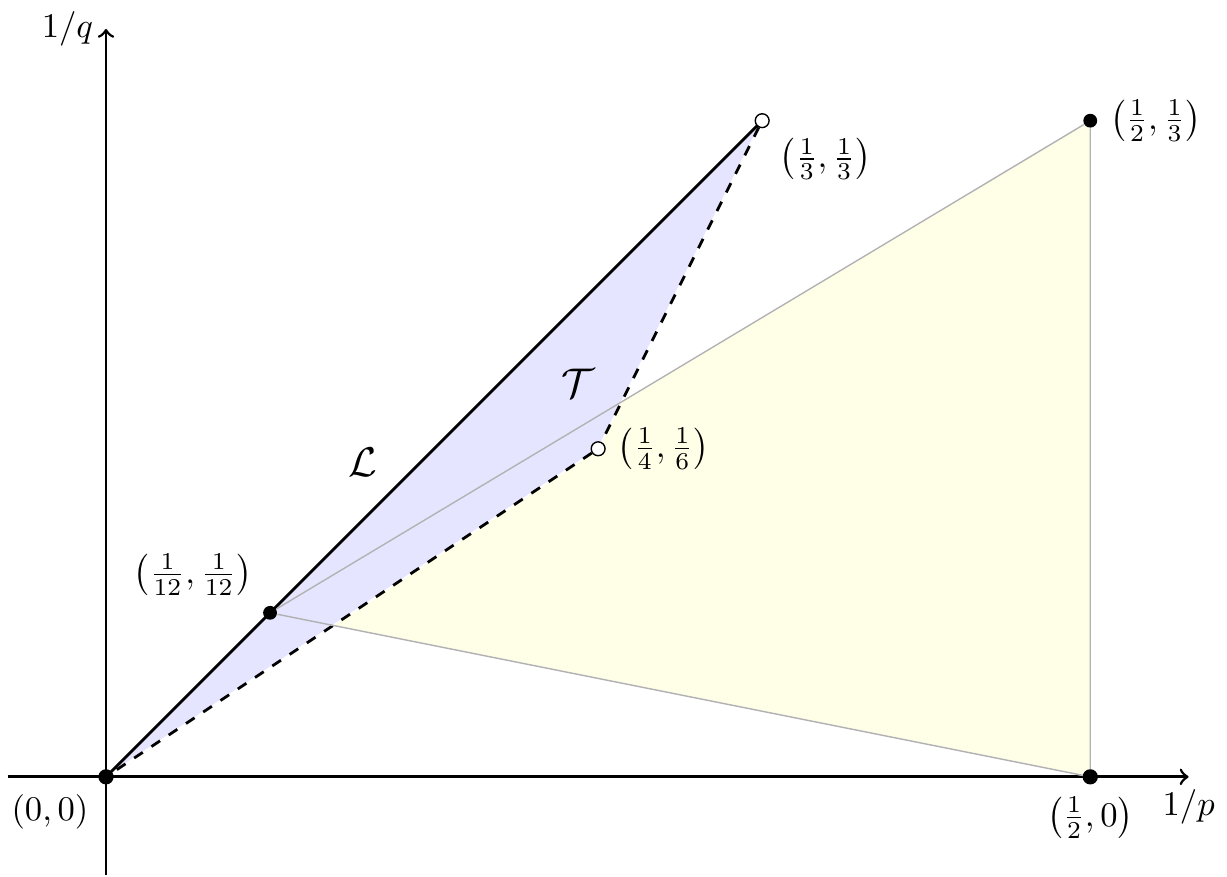}
    \caption{The known range of $L^p \to L^q$ boundedness for the helical maximal function. In~\cite{BGHS-helical} and~\cite{KLO2022}, boundedness was shown on the half-open line segment $\cL$ connecting $(0,0)$ and $(1/3, 1/3)$. By Theorem~\ref{thm:main}, the operator is bounded whenever $(1/p, 1/q) \in \mathrm{int}(\cT)$, the interior of the triangle with vertices $(0,0)$, $(1/3, 1/3)$ and $(1/4, 1/6)$. Frequency localised estimates are obtained at the critical vertex $(1/4, 1/6)$ by interpolating multilinear inequalities at $(1/2, 1/3)$, $(1/12, 1/12)$ and $(1/2, 0)$; see \S\ref{subsec: microloc est} below.}
    \label{fig: Riesz}
\end{figure}




\subsection{Methodology}

Here we provide a brief overview of the ingredients of the proof of Theorem~\ref{thm:main} and the novel features of the argument. For fixed $t$, the averaging operators $A_t$ correspond to convolution with an appropriate measure $\mu_t$ on the $t$-dilate of $\gamma$. It is therefore natural to study these objects via the Fourier transform, which leads us to consider the multiplier
\begin{equation*}
    \hat{\mu}_t(\xi) = \int_{\R} e^{-it\inn{\gamma(s)}{\xi}} \chi(s)\,\ud s. 
\end{equation*}
Stationary phase can be used to compute the decay rate of this function in different directions in the frequency space. This involves analysing the vanishing of $s$-derivatives of the phase function $\phi(\xi, s) := \inn{\gamma(s)}{\xi}$. Following earlier work on the circular maximal function~\cite{MSS1992}, it is also useful to study the Fourier transform of $A_tf(x)$ in \textit{both} the $x$ and the $t$ variables. This leads us to consider the $4$-dimensional $(\xi, \tau)$ frequency space.  

Broadly speaking, this approach was taken in both works~\cite{KLO2022} and~\cite{BGHS-helical} to study the $L^p \to L^p$ mapping properties of $M_{\gamma}$. However, these papers focused on different geometrical aspects of the problem. In very rough terms, the analysis of~\cite{KLO2022} centres around a $3$-dimensional cone $\Gamma_3$ in $(\xi, \tau)$-space arising from the equations $\partial_s\phi(\xi,s) = 0$, $\tau=\phi(\xi,s)$. On the other hand, the analysis of~\cite{BGHS-helical} centres around a $2$-dimensional cone $\Gamma_2$ in the $(\xi, \tau)$-space arising from the system of equations $\partial_s\phi(\xi,s) = \partial_s^2\phi(\xi,s)= 0$, $\tau=\phi(\xi,s)$.

It seems difficult to obtain almost optimal $L^p \to L^q$ estimates using either the approach of~\cite{KLO2022} or of~\cite{BGHS-helical} in isolation; rather, it appears necessary to incorporate both geometries into the analysis. In order to do this, we apply a recent observation of Bejenaru~\cite{Bejenaru2019}, which provides a localised variant of the Bennett--Carbery--Tao multilinear restriction theorem~\cite{BCT2006}. We describe the relevant setup in detail in \S\ref{sec: loc mlr} below; moreover, in the appendix we relate the required localised estimates to the theory of Kakeya--Brascamp--Lieb inequalities from~\cite{BBFL2018}. The local multilinear restriction estimate allows us to work simultaneously with the $\Gamma_2$ and $\Gamma_3$ geometries, by considering the embedded cone $\Gamma_2$ as a localised portion of $\Gamma_3$. See Proposition~\ref{prop:L2 L3 trilinear} below.

On the other hand, the geometries of both $\Gamma_2$ and $\Gamma_3$ were previously exploited in a non-trivial manner in~\cite{KLO2023} and~\cite{BGHS-Sobolev} using the decoupling inequalities from~\cite{BDG2016}. This approach is inspired by earlier work of Pramanik--Seeger~\cite{PS2007}. Decoupling is effective for proving $L^p \to L^p$ bounds for large $p$; here it is used to provide a counterpoint for interpolation with the estimates obtained via local multilinear restriction. See Proposition~\ref{prop:L12} below.




\subsection{Notational conventions} Throughout the paper, $I$ denotes the interval $[-1,1]$. 

We let $\widehat{\R}$ denote the \textit{frequency domain}, which is the Pontryagin dual group of $\R$  understood here as simply a copy of $\R$. Given $f \in L^1(\R^d)$ and $g \in L^1(\widehat{\R}^d)$ we define the Fourier transform and inverse Fourier transform by
\begin{equation*}
    \hat{f}(\xi) := \int_{\R^d} e^{-i \inn{x}{\xi}} f(x)\,\ud x \quad \textrm{and} \quad \check{g}(x) := \frac{1}{(2\pi)^d}\int_{\widehat{\R}^d} e^{i \inn{x}{\xi}} g(\xi)\,\ud \xi,
\end{equation*}
respectively. For $m \in L^{\infty}(\widehat{\R}^d)$ we define the multiplier operator $m(D)$ which acts initially on Schwartz functions by
\begin{equation*}
    m(D)f(x) := \frac{1}{(2\pi)^d} \int_{\widehat{\R}^d} e^{i \inn{x}{\xi}} m(\xi) \hat{f}(\xi)\,\ud \xi. 
\end{equation*}

Given a list of objects $L$ and real numbers $A$, $B \geq 0$, we write $A \lesssim_L B$ or $B \gtrsim_L A$ to indicate $A \leq C_L B$ for some constant $C_L$ which depends only items in the list $L$ and our choice of underlying non-degenerate curve $\gamma$. We write $A \sim_L B$ to indicate $A \lesssim_L B$ and $B \lesssim_L A$. 

\subsection{Organisation of the paper}
\begin{itemize}
    \item In \S\ref{sec: loc mlr} we present the key localised trilinear restriction estimate.
    \item In \S\ref{sec: init_red} we describe a reduction of Theorem~\ref{thm:main} to three local-smoothing-type estimates: trilinear $L^2 \to L^3$, linear $L^{12} \to L^{12}$ and trivial $L^2 \to L^{\infty}$.
    \item In \S\ref{sec: preliminaries} we describe the basic properties of our operators and prove the trivial $L^2 \to L^{\infty}$ estimate.
    \item In \S\ref{sec: trilinear} we prove the trilinear $L^2\rightarrow L^3$ estimate using the trilinear restriction theorem from \S\ref{sec: loc mlr}.
    \item In \S\ref{sec: L12 loc smoothing} we prove the linear $L^{12}\rightarrow L^{12}$ estimate using decoupling.
    \item In \S\ref{sec: non-deg} we bound a non-degenerate portion of the operator.
    \item In \S\ref{sec: trilin to lin} we carry out the reduction described in \S\ref{sec: init_red} and thereby bound the remaining degenerate portion of the operator. 
    \item In \S\ref{sec:nec} we demonstrate the sharpness of the range $\mathcal{T}$.
    \item Finally, in Appendix A we present a proof of the localised trilinear restriction theorem from \S\ref{sec: loc mlr}.
\end{itemize}

\subsection*{Acknowledgements} The first and third authors would like to thank Shaoming Guo and Andreas Seeger for discussions related to the topic of this paper over the years.




\section{Localised trilinear restriction}\label{sec: loc mlr}

The key ingredient in the proof of Theorem~\ref{thm:main} is a localised trilinear Fourier restriction estimate. Here we describe the particular setup for our problem. As in~\cite{KLO2022}, it is necessary to work with functions with a limited degree of regularity. 

\begin{definition} Let $0 < \alpha \leq 1$ and $U \subseteq \R^d$ be an open set. We say a function $Q \colon U \to \R$ is of class $C^{1,\alpha}(U)$ if it is continuously differentiable on $U$ and, moreover, the partial derivatives satisfy the $\alpha$-H\"older condition
\begin{equation*}
    \sup_{\substack{\xi_1, \xi_2 \in U \\ \xi_1 \neq \xi_2}} \frac{|\nabla Q(\xi_1) - \nabla Q(\xi_2)|}{|\xi_1 - \xi_2|^{\alpha}} < \infty. 
\end{equation*}
\end{definition}

Consider an ensemble $\mathbf{Q} = (Q_1, Q_2, Q_3)$ of maps $Q_j \colon U_j \to \R$ of class $C^{1,1/2}(U_j)$ where $U_j \subseteq \widehat{\R}^3$ is an open domain\footnote{Here an \textit{open domain} in $\widehat{\R}^d$ is an open, bounded, connected subset of $\widehat{\R}^d$.} for $1 \leq j \leq 3$. The graphs 
\begin{equation*}
  \Sigma_j := \{(\xi, Q_j(\xi)) : \xi \in U_j \}   
\end{equation*}
are hypersurfaces in $\widehat{\R}^4$, with some limited regularity. Each $\Sigma_j$ has a Gauss map given by 
\begin{equation*}
\nu_j \colon U_j \to S^3, \qquad \nu_j(\xi) := \frac{1}{(1 + |\nabla Q_j(\xi)|^2)^{1/2}} \begin{pmatrix}
    -\nabla Q_j(\xi) \\
    1
\end{pmatrix}
\qquad \textrm{for all $\xi \in U_j$.}
\end{equation*}
We further fix a smooth function $u \colon U_3 \to \R$ satisfying $|\nabla u(\xi)| > c_0 > 0$ for all $\xi \in U_3$. This implicitly defines a smooth surface $Z_3 := \{\xi \in \widehat{\R}^3 : u(\xi) = 0\}$, which we lift to
\begin{equation*}
    \Sigma_3' := \{(\xi, Q_3(\xi)) : \xi \in Z_3 \}.
\end{equation*}
Thus, $\Sigma_3'$ is a codimension 1 submanifold of $\Sigma_3$, which is embedded in $\widehat{\R}^4$. Defining
\begin{equation*}
  \nu_3' \colon Z_3 \to S^3, \quad  \nu_3'(\xi) := \frac{1}{|\nabla u(\xi)|} \begin{pmatrix}
        \nabla u(\xi) \\
        0
    \end{pmatrix} \qquad \textrm{for all $\xi \in Z_3$,}
\end{equation*}
it follows that $\{\nu_3(\xi), \nu_3'(\xi)\}$ forms a basis of the normal space to $\Sigma_3'$ at $(\xi, Q_3(\xi))$ for all $\xi \in Z_3$. 

We now fix $a_j \in C_c(U_j)$ with $\|a_j\|_{L^{\infty}(U_j)} \leq 1$ for $1 \leq j \leq 3$ and we assume the \textit{transversality hypothesis} 
\begin{equation}\label{eq: trans}
    \bigg|
    \det
    \begin{pmatrix}
        \nabla Q_1(\xi_1) & \nabla Q_2(\xi_2) & \nabla Q_3(\xi_3) & \nabla u(\xi_3) \\
        1 & 1 & 1 & 0
    \end{pmatrix}
       \bigg| 
    > c_{\mathrm{trans}} > 0 
\end{equation}
 for all $\xi_j \in \supp a_j$, $1 \leq j \leq 3$. Furthermore, given $0 < \mu < 1$, we assume the additional \textit{localisation hypothesis}
 \begin{equation}\label{eq: localisation}
 |u(\xi)| < \mu \qquad \textrm{for all $\xi \in \supp a_3$.} 
 \end{equation}
 Finally, we define the \textit{extension operators}
 \begin{equation*}
 E_jf(x,t) := \int_{\widehat{\R}^3} e^{i(\inn{x}{\xi} + tQ_j(\xi))} a_j(\xi) f(\xi)\,\ud \xi \qquad \textrm{for $f \in L^1(U_j)$, $1 \leq j \leq 3$.}    
 \end{equation*}
The key localised trilinear estimate is as follows.

\begin{theorem}[Localised trilinear restriction]\label{thm: local BCT} With the above setup, for all $\varepsilon > 0$ and all $R \geq 1$ we have
\begin{equation*}
    \Big\|\prod_{j=1}^3|E_jf_j|^{1/3}\Big\|_{L^3(B(0,R))} \lesssim_{\mathbf{Q}, \varepsilon} R^{\varepsilon} \mu^{1/6} \prod_{j=1}^3 \|f_j\|_{L^2(U_j)}^{1/3}
\end{equation*}
for all $f_j \in L^1(U_j)$, $1 \leq j \leq 3$.
\end{theorem}

Here the implied constant depends on the choice of maps $Q_j$ and, in particular, the lower bound in \eqref{eq: trans}, but is (crucially) independent of the choice of parameter $\mu$ in \eqref{eq: localisation} and the choice of scale $R$.

If we consider smooth hypersurfaces rather than the $C^{1,1/2}$ class, then Theorem~\ref{thm: local BCT} is a special case of~\cite[Theorem 1.3]{Bejenaru2019}. We expect that the arguments of~\cite{Bejenaru2019} can be generalised to treat $C^{1,\alpha}$ regularity for all $\alpha > 0$. However, in Appendix~\ref{sec: appendix} we observe that Theorem~\ref{thm: local BCT} is a rather direct consequence of the Kakeya--Brascamp--Lieb inequalities from~\cite{BBFL2018} (see also \cite{Zhang2018,Zorin-Kranich2020}).




\section{Initial reductions}\label{sec: init_red}




\subsection{Local smoothing estimates}

The multipliers of interest are of the following form. For $I := [-1,1]$, let $\gamma \colon I \to \R^3$ be a smooth curve and fix $\rho \in C^\infty_c(\R)$ supported in the interior $[1/2,4]$. Given a symbol $a \in C^{\infty}(\widehat{\R}^3\setminus\{0\} \times \R \times \R )$, we define
\begin{equation}\label{eq: m def}
    m_\gamma[a](\xi; t) \equiv m[a](\xi;t)  := \int_{\R} e^{-i t \inn{\gamma(s)}{\xi}} a(\xi;t; s) \psi_I(s)\rho(t)\,\ud s,
\end{equation}
for some $\psi_I \in C^{\infty}(\R)$ with support lying in $I$. Fix $\eta \in C^\infty_c(\R)$ non-negative, even and such that
\begin{equation*}
\eta(r) = 1 \quad \textrm{if $r \in I$} \quad \textrm{and} \quad \supp \eta \subseteq [-2,2]
\end{equation*}
and define $\beta$, $\beta^k \in C^\infty_c(\R)$ by
\begin{equation}\label{eq: beta def}
  \beta(r) := \eta(r) - \eta(2r)  \quad \textrm{and} \quad \beta^k(r) := \beta(2^{-k}r) \qquad \textrm{for all $k \in \Z$.}
\end{equation} 
By an abuse of notation, we also write $\eta(\xi) := \eta(|\xi|)$ and $\beta^k(\xi) := \beta^k(|\xi|)$ for $\xi \in \widehat{\R}^3$.

 For $a \in C^{\infty}(\widehat{\R}^3 \setminus \{0\} \times \R \times \R)$ as above, we form a dyadic decomposition by writing
\begin{equation}\label{symbol dec}
    a = \sum_{k = 0}^{\infty} a_k \qquad \textrm{where} \qquad  a_k(\xi; t; s) :=   \left\{ \begin{array}{ll}
        a(\xi; t; s) \, \beta^k(\xi) & \textrm{for $k \geq 1$} \\
         a(\xi; t; s) \, \eta(\xi) & \textrm{for $k =0$}
     \end{array} \right. .
\end{equation}

With the above definitions, our main result is as follows. 

\begin{theorem}[$L^p \to L^q$ local smoothing]\label{thm: Lp Lq loc smoothing} Let $\gamma:I \to \R^3$ be a smooth curve and suppose $a \in C^{\infty}(\widehat{\R}^3\setminus \{0\} \times \R \times \R)$ satisfies the symbol condition
\begin{equation}\label{eq: LS symbol}
    |\partial_{\xi}^{\alpha}\partial_t^i \partial_s^j a(\xi;t;s)| \lesssim_{\alpha, i, j} |\xi|^{-|\alpha|} \qquad \textrm{for all $\alpha \in \N_0^3$ and $i$, $j \in \N_0$}
\end{equation}
and that
\begin{equation}\label{eq: LS non deg}
    \sum_{j=1}^3|\inn{\gamma^{(j)}(s)}{\xi}| \gtrsim |\xi| \qquad \text{ for all $(\xi;s) \in \xisupp a \times I$}.
\end{equation}
Then for all $(1/p, 1/q) \in \mathrm{int}(\cT)$ there exists some $\varepsilon(p,q) > 0$ such that
\begin{equation*}
    \Big(\int_1^2\|m[a_k](D;t)f\|_{L^q(\R^3)}^q\,\ud t\Big)^{1/q} \lesssim_{p,q} 2^{-k/q - \varepsilon(p,q) k}\|f\|_{L^p(\R^3)}
\end{equation*}
holds for all $k \in \N_0$, where $a_k$ is defined as in \eqref{symbol dec}. 
\end{theorem}

The desired maximal bound follows from Theorem~\ref{thm: Lp Lq loc smoothing} using a standard Sobolev embedding argument; we omit the details but refer the reader to~\cite[Chapter XI, \S3]{Stein1993},~\cite[\S6]{PS2007} or~\cite[\S 2]{BGHS-helical} for similar arguments. 

By results of~\cite{BGHS-helical}, Theorem~\ref{thm: Lp Lq loc smoothing} is known to hold along the diagonal line $\cL$. By interpolation, it therefore suffices to prove an estimate at the critical vertex $(1/4, 1/6)$ in the Riesz diagram (see Figure~\ref{fig: Riesz}).

\begin{proposition}\label{prop:L4 to L6} Under the hypotheses of Theorem~\ref{thm: Lp Lq loc smoothing}, for $k \in \N_0$ and all $\varepsilon > 0$, we have
\begin{equation}\label{eq: critical est}
    \Big(\int_1^2\|m[a_k](D;t)f\|_{L^6(\R^3)}^6\,\ud t\Big)^{1/6} \lesssim_{\varepsilon} 2^{-k/6 + \varepsilon k}\|f\|_{L^4(\R^3)}.
\end{equation}
\end{proposition}

By the preceding discussion, our main theorem follows from Proposition~\ref{prop:L4 to L6}. Henceforth, we focus on the proof of this critical estimate.




\subsection{Trilinear reduction}\label{subsec: trilinear red} 
If the hypothesis \eqref{eq: LS non deg} is strengthened to
\begin{equation}\label{eq: super LS non deg}
    |\inn{\gamma'(s)}{\xi}| + |\inn{\gamma''(s)}{\xi}| \gtrsim |\xi| \qquad \textrm{for all $(\xi;s) \in \xisupp a \times I$,}
\end{equation}
then one can deduce the critical estimate \eqref{eq: critical est} as a consequence of known local smoothing inequalities from~\cite{PS2007} (see Theorem~\ref{PS LS J=2} below) and the Stein--Tomas Fourier restriction inequality. Given a small number $0 < \delta < 1$ and $k \in \N$, we perform this analysis on the symbols
\begin{equation}\label{eq: ak0 def}
    a_{k,0}(\xi;t;s):=a_k(\xi;t;s) \big(1-\eta (2^{-k}\delta^{-20} G_2(s;\xi)) \big)
\end{equation}
where $G_2(s;\xi):=\sum_{j=1}^2 |\inn{\gamma^{(j)}(s)}{\xi}|$; note that $a_{k,0}$  satisfies \eqref{eq: super LS non deg} with an implicit constant depending on $\delta$.  We discuss this case in detail in \S\ref{sec: non-deg}. 

The main difficulty is then to get to grips with the degenerate portion of the multiplier. For the above choice of $0<\delta<1$, this corresponds to the condition
\begin{equation}\label{eq: deg case}
    |\inn{\gamma'(s)}{\xi}| + |\inn{\gamma''(s)}{\xi}| \lesssim \delta^{20}|\xi|\qquad \textrm{for all $(\xi;s) \in \xisupp a \times I$;}
\end{equation}
note that this is satisfied on the support of $\mathfrak{a}_k := a_k-a_{k,0}$. To control the degenerate part, we work with a trilinear variant of Proposition~\ref{prop:L4 to L6}, from which we deduce the corresponding linear estimate \eqref{eq: critical est} via a standard application of the \textit{broad-narrow} method from~\cite{Bourgain2011} (see also~\cite{Ham2014}).  

To describe the trilinear setup, we introduce some notation. For $0 < \delta < 1$ as above, let $\fJ(\delta)$ denote a covering of $I$ by essentially disjoint intervals of length $\delta$. Let $\fJ^{3, \mathrm{sep}}(\delta)$ denote the collection of all triples $\bJ = \{J_1, J_2, J_3\} \subset \fJ(\delta)$ which satisfy the separation condition $\dist(J_i, J_j) \geq 10\delta$ for $1 \leq i < j \leq 3$. Given a bounded interval $J \subseteq \R$, we let $\psi_J \in C^{\infty}_c(\R)$ satisfy $\supp \psi_J \subseteq J$ and $|\partial^N_s\psi_J(s)| \lesssim_N |J|^{-N}$ for all $N \in \N$. Similarly to \eqref{eq: m def}, given  a symbol $a \in C^{\infty}(\widehat{\R}^3\setminus\{0\} \times \R \times \R )$, we define the multipliers adapted to an interval $J \in \mathfrak{J}(\delta)$ by
\begin{equation*}
m^J_\gamma[a](\xi; t) \equiv m^J[a](\xi;t)  := \int_{\R} e^{-i t \inn{\gamma(s)}{\xi}} a(\xi;t; s) \psi_J (s) \rho(t)\,\ud s.
\end{equation*}
With this setup, we prove the following estimate.

\begin{proposition}[$L^4 \to L^6$ trilinear local smoothing] \label{prop:L4 to L6 trilinear} Let $k \in \N_0$, $\varepsilon > 0$ and $\delta > 0$. Under the hypotheses of Theorem~\ref{thm: Lp Lq loc smoothing} and further assuming \eqref{eq: deg case}, we have
\begin{equation*}
    \Big(\int_1^2\Big\| \prod_{J \in \bJ}|m^J[a_k](D;t)f_J|^{1/3}\Big\|_{L^6(\R^3)}^6\,\ud t\Big)^{1/6} \lesssim_{\varepsilon} \delta^{-O(1)} 2^{-k/6 + \varepsilon k}\prod_{J \in \bJ} \|f_J\|_{L^4(\R^3)}^{1/3}
\end{equation*}
whenever $\bJ \in \fJ^{3, \mathrm{sep}}(\delta)$ and $f_J \in \Sc(\R^3)$ for $J \in \bJ$. 
\end{proposition}

Here we use the notation $O(1)$ to denote an unspecified absolute constant. In applications, we work with relatively large values of $\delta$ (namely, $\delta \sim_{\varepsilon} 1$) and accordingly there is no need to precisely track the $\delta$ dependence. We will also assume without loss of generality that $0 < \delta < c$ where $c > 0$ is a small absolute constant, chosen to satisfy the forthcoming requirements of the argument, and $k$ is sufficiently large depending on $\delta^{-1}$. 

As mentioned above, the (ostensibly weaker) trilinear estimate in Proposition~\ref{prop:L4 to L6 trilinear} implies the linear estimate in Proposition~\ref{prop:L4 to L6} (under the additional assumption \eqref{eq: deg case}) using a variant of the procedure introduced in~\cite{Bourgain2011}. We postpone the details of this reduction to \S\ref{sec: trilin to lin} below.




\subsection{Reduction to perturbations of the moment curve}\label{subsec: pert moment}

At small scales, any non-degenerate curve can be thought of as a perturbation of an affine image of the moment curve $\gamma_{\circ}(s) :=(s, s^2/2, s^3/6)$. We refer to~\cite[\S 4]{BGHS-helical} for details (which involve the affine rescalings described in \S\ref{subsec: scaling} below), and just record here that it suffices to consider curves in the class $\fG(\delta_0)$ defined below for $0 < \delta_0 <1$ sufficiently small.  

\begin{definition} Given $0 < \delta_0 < 1$ and $M \in \N$, let $\fG(\delta_0, M)$ denote the class of all smooth curves $\gamma \colon I \to \R^3$ that satisfy the following conditions: 
\begin{enumerate}[i)]
    \item $\gamma(0) = 0$ and $\gamma^{(j)}(0) = \vec{e}_j$ for $1 \leq j \leq 3$;
    \item $\|\gamma - \gamma_{\circ}\|_{C^M(I)} \leq \delta_0$ for all $0 \leq j \leq M$.
\end{enumerate}
Here $\vec{e}_j$ denotes the $j$th standard Euclidean basis vector and
\begin{equation*}
    \|\gamma\|_{C^M(I)} := \max_{1 \leq j \leq M} \sup_{s \in I} |\gamma^{(j)}(s)| \qquad \textrm{for all $\gamma \in C^M(I;\R^3)$.}
\end{equation*}
If $M = 4$, then we will simply write $\fG(\delta_0)$ for $\fG(\delta_0, 4)$
\end{definition}

Henceforth we will always assume that $\gamma \in \fG(\delta_0)$ for $\delta_0 := 10^{-10}$.




\subsection{Microlocal decomposition}\label{subsec: microloc dec}

Under the assumption \eqref{eq: deg case}, the non-degeneracy condition \eqref{eq:nondegenerate} ensures that
\begin{equation}\label{eq: 3rd order}
    |\inn{\gamma'''(s)}{\xi}| \gtrsim |\xi| \quad \text{ for all $(\xi;t;s) \in \supp a$;}
\end{equation}
indeed, for $\gamma \in \fG(\delta_0)$ this condition holds for all $(\xi;t;s) \in \xisupp a \times \R \times I$. We can then assume that $s \mapsto \inn{\gamma'''(s)}{\xi}$ has constant sign and henceforth we assume that $\xi_3>0$. 
Following~\cite[\S 6]{BGHS-helical},  let $\theta_2 \colon \xisupp a  \to I$ be the smooth mapping such that
\begin{equation*}
    \inn{\gamma'' \circ \theta_2(\xi)}{\xi} = 0 \qquad \textrm{for all $\xi \in \xisupp a$.}
\end{equation*}
It is clear that $\theta_2$ is homogeneous of degree 0. Let 
\begin{equation}\label{eq: u def}
    u(\xi) := \inn{\gamma' \circ \theta_2(\xi)}{\xi} \qquad \textrm{for all $\xi \in \xisupp a$.}
\end{equation}
Since \eqref{eq: 3rd order} is satisfied on the support of each 
\begin{equation}\label{eq: fa symbols}
    \fa_k:=a_{k}-a_{k,0},
\end{equation}
 we decompose each of these pieces with respect to the size of $|u(\xi)|$. Given $\varepsilon >0$ and $0 < \delta < 1$, we write
\begin{equation}\label{eq: a_kl dec}
   a_k =  a_{k, 0} + \sum_{\ell \in \Lambda(k)} a_{k,\ell}  + a_{k, k/3}
\end{equation}
where $a_{k,0}$ is as in \eqref{eq: ak0 def} and
\begin{equation}\label{eq: a_kl def}
    a_{k,\ell}(\xi; t; s) := 
    \left\{\begin{array}{ll}
        \displaystyle \fa_k(\xi; t; s) \, \beta\big(2^{-k+ 2\ell}u(\xi)\big) \quad   & \textrm{if $\ell \in \Lambda(k)$,}  \\[8pt]
        \displaystyle \fa_{k}(\xi; t; s)\, \eta\big(2^{-k + 2\floor{(1-2\varepsilon)k/3}}u(\xi)\big) & \textrm{if $\ell = k/3$}
    \end{array}\right.
\end{equation}
for
\begin{equation}\label{eq: def Lambda}
    \Lambda(k):= \{ \, \ell \in \N :  \lceil \log_2( \delta^{-8}) \rceil < \ell < \left\lfloor (1-2\varepsilon)k/3 \right\rfloor \,  \}.
\end{equation}
Here we assume that $k$ is large enough so that the decomposition \eqref{eq: a_kl dec} makes sense; note that Theorem \ref{thm: Lp Lq loc smoothing} trivially holds for small values of $k$. In particular, we concern ourselves with $k \in \N$ satisfying $k \geq 4 \log_2(\delta^{-8})$. In the definition \eqref{eq: def Lambda}, for any $x \in \R$, we let $\floor{x}$  denote the largest integer less or equal than $x$ and $\ceil{x}$ denote the smallest integer greater or equal than $x$. It will also be useful to introduce the notation
\begin{equation*}
    \overline{\Lambda}(k) := \Lambda(k) \cup \{k/3\}.
\end{equation*}
Note that the indexing set $\Lambda(k)$ depends on the chosen $\delta$ and $\varepsilon$, but we do not record this dependence for notational convenience. 
We also note that here the function $\beta$ should be defined slightly differently compared with \eqref{eq: beta def}; in particular, here $\beta(r) := \eta(2^{-2}r) - \eta(r)$ (we ignore this minor change in the notation).

As mentioned in \S\ref{subsec: trilinear red}, for the extreme case $\ell=0$ we have the non-degeneracy condition \eqref{eq: super LS non deg} (with an implied constant depending on $\delta$). This situation is easy to handle using known estimates: see \S\ref{sec: non-deg} below. On the other hand, for $\ell \in \overline{\Lambda}(k)$, the multipliers $m[a_{k,\ell}]$ are degenerate in the sense that \eqref{eq: deg case} now holds. A key aspect of this decomposition is that for $\ell \in \Lambda(k)$, Taylor series expansion shows that
\begin{equation}\label{eq: weak non-deg}
    |\inn{\gamma'(s)}{\xi}| + 2^{-\ell}|\inn{\gamma''(s)}{\xi}| \gtrsim 2^{k-2\ell} \qquad \textrm{for all $(\xi, s) \in \supp a_{k,\ell}(\,\cdot\,;t;\,\cdot\,)$;}
\end{equation}
see, for example~\cite[(5.15)]{BGHS-Sobolev} for a detailed derivation. The weak non-degeneracy condition \eqref{eq: weak non-deg} will allow for improved estimates depending on the value of $\ell$. Bounding these pieces, and the piece for $\ell=k/3$, is the difficult part of the argument and is the focus of \S\S\ref{sec: trilinear}--\ref{sec: L12 loc smoothing} below.




\subsection{Microlocalised estimates}\label{subsec: microloc est} Throughout this section we work under the hypotheses of Theorem~\ref{thm: Lp Lq loc smoothing} and, in addition, assume \eqref{eq: deg case} holds for a specified value of $\delta$. That is, we let $\gamma: I \to \R^3$ be a smooth curve and suppose $a \in C^{\infty}(\widehat{\R}^3\setminus \{0\} \times \R \times \R)$ satisfies \eqref{eq: LS symbol}, \eqref{eq: LS non deg} and \eqref{eq: deg case}. Furthermore, we define the symbols $a_{k,\ell}$ as in \eqref{eq: a_kl def}.

The key ingredient in the proof of Proposition~\ref{prop:L4 to L6 trilinear} is a trilinear estimate for the multipliers associated to the localised symbols $a_{k,\ell}$. To describe this result, we work with a triple of integers $\ell_{\bJ} = (\ell_J)_{J \in \bJ}$ indexed by $\bJ \in \fJ^{3, \mathrm{sep}}(\delta)$ and write $|\ell_{\bJ}| := \sum_{J \in \bJ} \ell_J$. 

\begin{proposition}[$L^{2} \to L^{3}$ trilinear local smoothing]\label{prop:L2 L3 trilinear}
    For $k \in \N$, $\varepsilon>0$ and $0<\delta < 1$, we have 
\begin{equation*}
    \Big(\int_1^2\Big\| \prod_{J \in \bJ}|m^J[a_{k,\ell_J}](D;t)f_J|^{1/3}\Big\|_{L^3(\R^3)}^3\,\ud t\Big)^{1/3} \lesssim_{\varepsilon} \delta^{-O(1)} 2^{-k/3 + |\ell_{\bJ}|/18 + \varepsilon k}\prod_{j=1}^3 \|f_J\|_{L^2(\R^3)}^{1/3}
\end{equation*}
whenever $\bJ \in \fJ^{3, \mathrm{sep}}(\delta)$, $\ell_{\bJ} = (\ell_J)_{J \in \bJ}$ with $ \ell_J \in \overline{\Lambda}(k)$ and $f_J \in \Sc(\R^3)$ for $J \in \bJ$. 
\end{proposition}

Proposition~\ref{prop:L2 L3 trilinear} is a fairly direct consequence of Theorem~\ref{thm: local BCT}; we describe the proof in \S\ref{sec: trilinear} below. To deduce the critical $L^4 \to L^6$ estimate stated in Proposition~\ref{prop:L4 to L6 trilinear}, we interpolate Proposition~\ref{prop:L2 L3 trilinear} with the following linear inequalities. 

\begin{proposition}[$L^{12} \to L^{12}$ local smoothing]\label{prop:L12}
For $k \in \N$, $\varepsilon>0$, $0 <\delta < 1$  and $ \ell \in \overline{\Lambda}(k)$, we have
\begin{equation*}
    \Big(\int_1^2 \| m[a_{k,\ell }](D;t)f \|_{L^{12}(\R^3)}^{12}\,\ud t\Big)^{1/12} \lesssim_{\varepsilon}  2^{-k/6 -\ell/12 + \varepsilon k}\|f\|_{L^{12}(\R^3)}.
\end{equation*}    
\end{proposition}

\begin{lemma}[$L^{2} \to L^{\infty}$ estimate]\label{lem:L2 Linfty}
For $k \in \N$, $\varepsilon>0$, $0 <\delta < 1$  and $ \ell \in \overline{\Lambda}(k)$, we have
\begin{equation*}
    \sup_{1 \leq t \leq 2} \|  m[a_{k,\ell}](D;t)f \|_{L^\infty(\R^3)}  \lesssim  2^{k -  \ell/2}\|f\|_{L^2(\R^3)}.
\end{equation*}
\end{lemma}
We remark that $0< \delta < 1$ plays no significant rôle in the proofs of Proposition \ref{prop:L12} and Lemma \ref{lem:L2 Linfty} and it is used only to set up the underlying decomposition in the $a_{k,\ell}$. Similarly, $\varepsilon>0$ plays no significant rôle in Lemma \ref{lem:L2 Linfty}. 

Proposition~\ref{prop:L12} is a minor variant of estimates which have appeared in, for instance,~\cite{KLO2023} and~\cite{BGHS-Sobolev}. The result is highly non-trivial, and relies on the $\ell^p$ decoupling inequality for the moment curve from~\cite{BDG2016}. We discuss the details in \S\ref{sec: L12 loc smoothing}.

Lemma~\ref{lem:L2 Linfty}, on the other hand, is elementary. It follows from basic pointwise estimates for the multipliers $a_{k,\ell}$, obtained via stationary phase. We discuss the details in \S\ref{subsec: L2 Linfty proof}.

Given the preceding results, the key trilinear $L^4 \to L^6$ local smoothing estimate is immediate.

\begin{proof}[Proof (of Proposition~\ref{prop:L4 to L6 trilinear})]
By multilinear Hölder's inequality, Propositions \ref{prop:L12} and \ref{lem:L2 Linfty} imply their trilinear counterparts. Since 
\begin{equation*}
\begin{pmatrix}
\frac{1}{4} \\[3pt]
\frac{1}{6}
\end{pmatrix}
=
\frac{3}{5} \cdot
\begin{pmatrix}
\frac{1}{12} \\[3pt]
\frac{1}{12}
\end{pmatrix}
+
\frac{7}{20} \cdot
\begin{pmatrix}
\frac{1}{2} \\[3pt]
\frac{1}{3}
\end{pmatrix}
+
\frac{1}{20} \cdot
\begin{pmatrix}
\frac{1}{2} \\[3pt]
0
\end{pmatrix},
\end{equation*}
interpolation of the three estimates immediately gives
\begin{equation*}
    \Big(\int_1^2\Big\| \prod_{J \in \bJ} |m^J[a_{k,\ell_J}](D;t)f|^{1/3}\Big\|_{L^6(\R^3)}^6\,\ud t\Big)^{1/6} \lesssim_{\varepsilon} \delta^{-O(1)} 2^{-k/6 - |\ell_{\bJ}|/180 + \varepsilon k}\|f\|_{L^4(\R^3)}
\end{equation*}
for $\ell_{\bJ} = (\ell_J)_{J \in \bJ}$ with $0 \leq \ell_J \leq k/3$; see Figure~\ref{fig: Riesz}. Here we carry out the interpolation using a multilinear variant of the Riesz--Thorin theorem: see, for instance, \cite[\S4.4]{BL1976}.\footnote{Alternatively, a suitable multilinear interpolation theorem can be proved by directly adapting the argument used to prove the classical Riesz--Thorin theorem.} Using the geometric decay in $2^{-\ell_J}$ for each $J \in \bJ$, we sum these bounds to deduce the desired result. 
\end{proof}

To prove Proposition~\ref{prop:L4 to L6 trilinear}, it therefore remains to establish Proposition~\ref{prop:L2 L3 trilinear}, Proposition~\ref{prop:L12} and Lemma~\ref{lem:L2 Linfty}. We carry this out in \S\S\ref{sec: preliminaries}--\ref{sec: L12 loc smoothing} below.




\section{Basic properties of the multipliers}\label{sec: preliminaries}




\subsection{Elementary estimates for the multiplier}\label{subsec: L2 Linfty proof}

Using stationary phase arguments, we can immediately deduce Lemma~\ref{lem:L2 Linfty}. 

\begin{proof}[Proof (of Lemma~\ref{lem:L2 Linfty})] By the Cauchy--Schwarz inequality, we have the elementary inequality
\begin{equation*}
    \|m(D)f\|_{L^{\infty}(\R^3)} \leq \|m\|_{L^{\infty}(\widehat{\R}^3)} \big|\supp m\big|^{1/2}\|f\|_{L^2(\R^3)}.
\end{equation*}
Fixing $t \in \R$, in view of the above it suffices to show
\begin{equation*}
    \|m[a_{k,\ell}](\,\cdot\,,t)\|_{\infty} \lesssim 2^{-(k-\ell)/2}, \quad |\supp m[a_{k,\ell}](\,\cdot\,;t)| \lesssim 2^{3k - 2\ell}.
\end{equation*}
Since $\nabla u(\xi) = \gamma'\circ \theta_2(\xi)$ is bounded away from zero, the latter estimate is clear. On the other hand, the former estimate is a consequence of a simple stationary phase analysis. Indeed, for $\ell = k/3$ we apply van der Corput's inequality with third order derivatives. For $\ell  \in \Lambda(k)$, we apply van der Corput's inequality with either first or second order derivatives, using the lower bound \eqref{eq: weak non-deg}. For further details see, for example,~\cite[Lemma 3.3]{PS2007},~\cite[(5.15)]{BGHS-Sobolev} or Lemma~\ref{lem: stationary phase} below.
\end{proof}




\subsection{Scaling of the multiplier}\label{subsec: scaling}

Let $\sigma \in I$, $0 < \lambda < 1$ be such that $[\sigma - \lambda, \sigma+\lambda] \subseteq I$. Denote by $[\gamma]_{\sigma}$ the $3\times 3$ matrix
\begin{equation*} 
    [\gamma]_{\sigma}:=
    \begin{bmatrix}
    \gamma^{(1)}(\sigma) & \gamma^{(2)} (\sigma)& \gamma^{(3)}(\sigma)
    \end{bmatrix},
\end{equation*}
where the vectors $\gamma^{(j)}(\sigma)$ are understood to be column vectors. Note that this matrix is invertible due to the non-degeneracy hypothesis \eqref{eq:nondegenerate}. It is also convenient to let $[\gamma]_{\sigma,\lambda}$ denote the $3 \times 3$ matrix
\begin{equation}\label{gamma transformation}
[\gamma]_{\sigma,\lambda} := [\gamma]_{\sigma} \cdot D_{\lambda}
\end{equation}
where $D_{\lambda}:=\text{diag}(\lambda, \lambda^2, \lambda^3)$ is the diagonal matrix with eigenvalues $\lambda$, $\lambda^2, \lambda^3$. With this data, 
define the \textit{$(\sigma,\lambda)$-rescaling of $\gamma$} as the curve $\gamma_{\sigma,\lambda} \in C^{\infty}(I;\R^{3})$ given by
\begin{equation}\label{eq: curve rescale}
    \gamma_{\sigma,\lambda}(s) := [\gamma]_{\sigma,\lambda}^{-1}\big( \gamma(\sigma+\lambda s) - \gamma(\sigma) \big).
\end{equation}
A simple computation shows 
\begin{equation*}
  \det [\gamma_{\sigma,\lambda}]_{s}  = \frac{\det\big( [\gamma]_{\sigma + \lambda s}\big)}{\det \big([\gamma]_{\sigma}\big)},
\end{equation*}
and therefore $\gamma_{\sigma, \lambda}$ is also a non-degenerate curve. Furthermore, the rescaled curve satisfies the relations
\begin{equation}\label{eq:inner prods scaled}
    \inn{\gamma_{\sigma,\lambda}^{(j)}(s)}{\xi}= \lambda^{j} \inn{\gamma^{(j)}(\sigma + \lambda s)}{([\gamma]_{\sigma, \lambda})^{-\top} \xi} \qquad \text{for all $j \geq 1$}.
\end{equation}
Combining this with the fact that $\gamma_{\sigma, \lambda}$ is non-degenerate, we have
 \begin{equation}\label{eq: resc symb supp}
     |\xi| \lesssim \sum_{j=1}^3 |\inn{\gamma_{\sigma,\lambda}^{(j)}(s)}{\xi}| \lesssim \lambda^3|([\gamma]_{\sigma, \lambda})^{-\top} \xi| \lesssim |\xi|;
 \end{equation}
 here the last inequality is a simple consequence of the definition \eqref{gamma transformation}.  
 

Defining the rescaled symbol
\begin{equation}\label{eq: symbol rescale}
a_{\sigma, \lambda}(\xi; t; s)  := a(([\gamma]_{\sigma, \lambda})^{-\top} \xi; t; \sigma + \lambda s),
\end{equation}
a change of variables immediately yields
\begin{equation}\label{eq:scaling mult identity}
    m_{\gamma}[a](D; t) f(x) = \lambda \,  m_{\gamma_{\sigma, \lambda}}[a_{\sigma, \lambda}](D; t) f_{\sigma, \lambda}(([\gamma]_{\sigma, \lambda})^{-1}(x-t\gamma(\sigma)))
\end{equation}
where $f_{\sigma, \lambda}:= f \circ [\gamma]_{\sigma, \lambda}$. In particular, by scaling,
\begin{equation}\label{eq: mult norm scaled}
    \| m_\gamma[a](D;\,\cdot\,) f \|_{L^p(\R^{3}) \to L^q(\R^{3+1})} \lesssim  \delta^{1+6(\frac{1}{q}-\frac{1}{p})} \| m_{\gamma_{\sigma, \lambda}}[a_{\sigma, \lambda}](D; \cdot) \|_{L^p(\R^3) \to L^{q}(\R^{3+1})}. 
\end{equation}

Finally, we observe that if the original symbol satisfies the condition
\begin{equation}\label{eq: LS symbol recall}
    |\partial_{\xi}^{\alpha} a(\xi;t;s)| \lesssim_{\alpha} |\xi|^{-|\alpha|} \qquad \textrm{for all $\alpha \in \N_0^3$},
\end{equation}
then so too does the rescaled symbol $a_{\sigma, \lambda}$. Indeed, by the definitions \eqref{eq: symbol rescale} and \eqref{gamma transformation}, we have
\begin{equation*}
    |(\partial_{\xi}^{\alpha}a_{\sigma, \lambda})(\xi;t;s)| \lesssim_{\alpha} \lambda^{-3|\alpha|} |(\partial_{\xi}^{\alpha} a)(([\gamma]_{\sigma, \lambda})^{-\top} \xi; t; \sigma + \lambda s)|.
\end{equation*}
In view of the hypothesis \eqref{eq: LS symbol recall} and \eqref{eq: resc symb supp}, it follows that
\begin{equation*}
    |(\partial_{\xi}^{\alpha} a_{\sigma, \lambda})(\xi;t;s)| \lesssim_{\alpha} \lambda^{-3|\alpha|} \big|([\gamma]_{\sigma, \lambda})^{-\top} \xi\big|^{-|\alpha|} \lesssim_{\alpha} |\xi|^{-|\alpha|}, 
\end{equation*}
as required. 




\subsection{Critical points} 

We next describe the critical points of the phase function $s \mapsto  \inn{\gamma(s)}{\xi}$ which, under the setup of \S\ref{subsec: microloc dec}, depend on the sign of the quantity $u(\xi)$ introduced in \eqref{eq: u def}.

\begin{lemma}[{\cite[Lemma 6.2]{BGHS-helical}}]\label{theta1 lem}
Let $\xi \in \xisupp a$ and consider the equation \begin{equation}\label{0404e3.29}
    \inn{\gamma'(s)}{\xi}=0.
\end{equation}
\begin{enumerate}[i)]
    \item If $u(\xi)>0$, then the equation \eqref{0404e3.29} has no solution on $I$.
\item If $u(\xi)=0$, then the equation \eqref{0404e3.29} has only the solution $s=\theta_2(\xi)$ on $I$.
\item If $u(\xi)<0$, then the equation \eqref{0404e3.29} has precisely two solutions on $I$.
\end{enumerate}
\end{lemma}

Following~\cite[\S 6]{BGHS-helical}, we can use Lemma~\ref{theta1 lem} to construct a (unique) pair of smooth mappings
\begin{equation*}
    \theta_1^{\pm} \colon \{ \xi \in \xisupp a :  u(\xi) <0 \} \to I
\end{equation*}
with $\theta_1^-(\xi)\le \theta_1^+(\xi)$ which satisfies
\begin{equation*}
    \inn{\gamma' \circ \theta_1^{\pm}(\xi)}{\xi}= 0 \quad \text{ for all $\xi \in \xisupp a$ with  $u(\xi) <0$.}
\end{equation*}
Define the functions 
\begin{equation*}
        v^{\pm}(\xi):=\inn{\gamma'' \circ \theta_1^\pm(\xi)}{\xi}  \qquad \text{ for all $\xi \in \xisupp a$ with  $u(\xi) <0$.}
\end{equation*}
Taylor expansion yields the following.

\begin{lemma}[{\cite[Lemma 6.3]{BGHS-helical}}]\label{lem: root}
Let $\xi \in \xisupp a$ with $u(\xi)<0$. Then
\begin{equation*}
    \big|v^{\pm}\big(\xip \big)\big| \sim |\theta_1^{\pm}(\xi) - \theta_2(\xi)| \sim  |\theta_1^+(\xi)-\theta_1^-(\xi)|\sim \big|u\big(\xip\big)\big|^{1/2}.
\end{equation*}
\end{lemma}




\subsection{Stationary phase}

We next use the approach in~\cite{KLO2022} and apply stationary phase to express the multipliers $m^J[a_{k,\ell}]$ as a product of a symbol and an oscillatory term. In what follows, we let \begin{equation*}
    q_2(\xi) := \inn{\gamma\circ\theta_2(\xi)}{\xi} \qquad \textrm{and} 
\qquad q_1^{\pm}(\xi) := \inn{\gamma\circ\theta_1^{\pm}(\xi)}{\xi}
\end{equation*} 
for any value of $\xi$ such that the expression is well-defined. Our analysis leads to various rapidly decreasing error terms. Given $R \geq 1$, we let $\mathrm{RapDec}(R)$ denote the class of functions $e \in C^{\infty}(\widehat{\R}^3 \times \R)$ which satisfy $|e(\xi;t)| \lesssim_N R^{-N}$ for all $N \in \N_0$.

\begin{lemma}\label{lem: stationary phase} Let $k \in \N$ and $J \in \mathfrak{J}(\delta)$.
\begin{enumerate}[i)]
    \item For some $e_{k,k/3} \in \mathrm{RapDec}(2^k)$, we may write
\begin{equation}\label{eq: stationary k/3}
    m^J[a_{k,k/3}](\xi;t) = 2^{-k/3} e^{-itq_2(\xi)} b^J_{k,k/3}(2^{-k}\xi;t) + e_{k,k/3}(\xi;t)
\end{equation}
where $b^J_{k,k/3} \in C^{\infty}(\R^{3+1})$ is supported in $B(0,10)$, satisfies $|u(\xi)| \lesssim 2^{-2k/3 + 4\varepsilon k/3}$ and $\dist(\theta_2(\xi), J) < \delta$ for $\xi \in \xisupp b^J_{k, k/3}$  and
\begin{equation}\label{eq: stationary k/3 der bounds}
    |\partial_t^N b^J_{k,k/3}(2^{-k} \xi;t)| \lesssim_N 2^{3k\varepsilon N + \varepsilon k} \qquad \textrm{for all $(\xi; t) \in \R^{3+1}$ and all $N \in \N_0$.}
\end{equation}
    \item For $\ell \in \Lambda(k)$ and some $e_{k,\ell} \in \mathrm{RapDec}(2^k)$, we may write
\begin{equation}\label{eq: stationary ell}
    m^J[a_{k,\ell}](\xi;t) = 2^{-(k-\ell)/2} \sum_{\pm} e^{-itq_1^{\pm}(\xi)} b^{J,\pm}_{k,\ell}(2^{-k}\xi;t) + e_{k,\ell}(\xi;t)
\end{equation}
where the $b^{J,\pm}_{k,\ell} \in C^{\infty}(\R^{3+1})$ are supported in $B(0,10)$, satisfy $|u(\xi)| \sim 2^{-2\ell}$ and $\dist(\theta_2(\xi), J) < \delta$ for $\xi \in \xisupp b^{J, \pm}_{k, \ell}$ and
\begin{equation}\label{eq: stationary ell der bounds}
    |\partial_t^N b^{J,\pm}_{k,\ell}(2^{-k} \xi;t)| \lesssim_N 1 \qquad \textrm{for all $(\xi; t) \in \R^{3+1}$ and all $N \in \N_0$.}
\end{equation}
\end{enumerate}
\end{lemma}

\begin{remark} In part ii) of the lemma, $u(\xi)<0$ always holds on the support of the $b^{J,\pm}_{k,\ell}$ and therefore, by Lemma \ref{theta1 lem}, the functions $\theta_1^{\pm}(\xi)$ are well-defined. The portion of the multiplier supported on the set where $u(\xi)>0$ is incorporated into the error term $e_{k,\ell}$.
\end{remark}

\begin{proof}     
i) By a change of variables,
\begin{equation*}
     m^J[a_{k,k/3}](\xi;t) = 2^{-k/3} e^{-itq_2(\xi)} \int_\R e^{- i t \Phi_k(\xi;s)} \alpha_k^J(\xi;t; 2^{-k/3}s) \, \ud s
\end{equation*}
where:
\begin{itemize}
    \item The symbol $\alpha_k^J \in C^{\infty}(\widehat{\R}^3\backslash \{0\} \times \R \times \R)$ satisfies
\begin{equation}\label{eq: alpha conds k/3}
 |u(\xi)| \lesssim 2^{k/3 + 4\varepsilon k/3}, \quad |\xi| \leq 2^{k+1}, \quad |t| \leq 4 \quad  \textrm{and} \quad \theta_2(\xi) + 2^{-k/3}s \in J 
\end{equation}
for $(\xi; t; 2^{-k/3}s) \in \supp \alpha_k^J$ and $|\partial_t^N \partial_s^M \alpha_k^J(\xi; t;s)| \lesssim_{M, N} 1$ for all $M$, $N \in \N_0$;
    \item The phase $\Phi_k$ is given by
\begin{equation*}
    \Phi_k(\xi;s) := \inn{\gamma(\theta_2(\xi) + 2^{-k/3}s)}{\xi} - \inn{\gamma \circ \theta_2(\xi)}{\xi}.
\end{equation*}
\end{itemize}

If $|s| \geq 2^{\varepsilon k}$ and $k$ is sufficiently large then, by combining \eqref{eq: alpha conds k/3} with a simple Taylor expansion argument, we see that $|\partial_s\Phi_k(\xi;s)| \gtrsim 2^{k\varepsilon}$. Therefore, by repeated integration-by-parts, we obtain \eqref{eq: stationary k/3} for 
\begin{equation*}
   b_{k,k/3}^J(2^{-k}\xi; t) : =   \int_\R e^{- i t \Phi_k(\xi;s)} \alpha_k^J(\xi;t; 2^{-k/3}s) \eta(2^{-\varepsilon k}s) \ud s
\end{equation*}
and some $e_{k,k/3} \in \mathrm{RapDec}(2^k)$.

By \eqref{eq: alpha conds k/3}, we have $\mathrm{dist}(\theta_2(\xi), J) \lesssim 2^{-k/3 + \varepsilon k} < \delta^2$ for $\xi \in \xisupp b_{k,k/3}^J$. On the other hand, \eqref{eq: stationary k/3 der bounds} now immediately follows for $N = 0$, using the triangle inequality. For larger values of $N$, the bounds follow from the estimate
\begin{equation*}
        |\Phi_k(\xi;s)| \lesssim 2^{3k\varepsilon} \qquad \text{ for all $(\xi, t; s) \in \supp \alpha_k^J$ with $|s| \lesssim 2^{\varepsilon k}$.} 
\end{equation*}
which is again a consequence of \eqref{eq: alpha conds k/3} and Taylor expansion. \medskip

\noindent ii) If $u(\xi)>0$, we know from Lemma \ref{theta1 lem} that the phase function $s \mapsto \inn{\gamma(s)}{\xi}$ has no critical points, and one can therefore obtain rapid decay of the portion of the multiplier where this condition holds; see~\cite[Lemma 8.1]{BGHS-helical} for similar arguments. We thus focus on the portion of the multiplier where $u(\xi) < 0$. Arguing analogously to the proof of part i), for a given $\rho>0$ we may define
\begin{equation*}
    b_{k,\ell}^{J,\pm}(2^{-k}\xi;t) := 2^{(k-\ell)/2} 2^{-\ell}  \int_\R e^{- i t 2^{k-3\ell}\Phi_{k,\ell}^{\pm}(\xi;s)} \alpha_{k,\ell}^{J,\pm}(\xi;t; 2^{-\ell}s) \eta(\rho^{-1}s) \, \ud s
\end{equation*}
where:
\begin{itemize}
    \item The symbols $\alpha_{k,\ell}^{J,\pm} \in C^\infty(\widehat{\R}^3 \backslash \{0\} \times \R \times \R)$ satisfy 
    \begin{equation*}
|u(\xi)| \sim 2^{k-2\ell}, \quad  |v^\pm(\xi)| \sim 2^{k-\ell}, 
\quad
|\xi| \leq 2^{k+1}, \quad |t| \leq 4, \quad \theta_1^{\pm}(\xi) + 2^{-\ell} s \in J
 \end{equation*}
for $(\xi;t;2^{-\ell}s) \in \supp \alpha_{k,\ell}^{J,\pm}$ and    $|\partial_t^N \partial_s^M \alpha_{k,\ell}^{J,\pm}(\xi; t;s)| \lesssim_{M, N} 1$ for all $M$, $N \in \N_0$;
\item  The phase $\Phi_{k,\ell}^\pm$ is given by
\begin{equation*}
    \Phi_{k,\ell}^\pm(\xi;s):=2^{-(k-3\ell)}\inn{\gamma(\theta_1^\pm(\xi) + 2^{-\ell} s) - \gamma \circ \theta_1^\pm(\xi)}{\xi}.
\end{equation*}
Moreover, $s=0$ is the only critical point of $s \mapsto \Phi_{k,\ell}^\pm(\xi;s)$ on the support of $\alpha_{k,\ell}^{J,\pm}$ if $|s| \leq \rho$ for $\rho>0$ sufficiently small.
\end{itemize}
An integration-by-parts argument similar to that in~\cite[Lemma 8.1]{BGHS-helical} then shows \eqref{eq: stationary ell} holds for some $e_{k,\ell} \in \mathrm{RapDec}(2^k)$. 

By the support properties of the $\alpha_{k,\ell}^{J, \pm}$ and Lemma~\ref{lem: root}, we have 
\begin{equation*}
   \mathrm{dist}(\theta_2(\xi), J) \leq |\theta_2(\xi) - \theta_1^{\pm}(\xi)| +\mathrm{dist}(\theta_1^{\pm}(\xi), J)  \lesssim 2^{-\ell} < \delta^2 \qquad \textrm{for $\xi \in \xisupp b_{k,\ell}^{J,\pm}$} 
\end{equation*}
and $\ell \in \Lambda(k)$. On the other hand, \eqref{eq: stationary ell der bounds} for $N=0$ follows from van der Corput's lemma, since  $|\partial_{ss}^2 \Phi_{k,\ell}^\pm (\xi;0)|= 2^{-k + \ell} |v^\pm(\xi)| \sim 1$ on $\supp \alpha_{k,\ell}^{J,\pm}$. For $N>0$, it suffices to show that
    \begin{equation}\label{eq: stationary ell N>0}
        \Big|\int_\R  e^{-i t 2^{k-3\ell} \Phi_{k,\ell}^\pm(\xi;s)} \big( 2^{k-3\ell} \Phi_{k,\ell}^\pm (\xi;s) \big)^N \alpha_{k,\ell}^{J,\pm}(\xi; t; 2^{-\ell} s) \, \ud s  \Big| \lesssim 2^{-(k-3\ell)/2}.
    \end{equation}
    By Taylor expansion, we have
        \begin{equation}\label{eq: stationary ell Phi 1}
        \Phi_{k,\ell}^\pm(\xi;s) =\frac{s^2}{2} \big[2^{-k+\ell}  v^\pm(\xi) + 2^{-k} \omega(\xi;s) s\big]
    \end{equation}
    where $|\omega(\xi;s)|\lesssim 2^{k}$ for  all $(\xi;t;s) \in \supp \alpha_{k,\ell}^{J,\pm}$; in particular, $|\Phi_{k,\ell}^\pm (\xi;s)| \lesssim 1$ in the support of $\alpha_{k,\ell}^{J,\pm}$. In view of the factor $s^2$ in \eqref{eq: stationary ell Phi 1}, the bound \eqref{eq: stationary ell N>0} is again a consequence of van der Corput's lemma  with second-order derivatives (in the specific form of, for example,~\cite[Lemma 1.1.2]{SoggeBook}).\medskip
\end{proof}




\section{Proof of the \texorpdfstring{$L^2 \to L^3$}{2} trilinear estimate}\label{sec: trilinear}

In this section we prove the key trilinear estimate from Proposition~\ref{prop:L2 L3 trilinear}. After massaging the operator into a suitable form, this is a consequence the localised multilinear restriction inequality from Theorem~\ref{thm: local BCT}.




\subsection{Reduction to multlinear restriction}\label{subsec: trilinear rest red}

Define the Fourier integral operators 
\begin{equation*}
    U^J_{k,k/3}f(x,t) := \int_{\widehat{\R}^3} e^{i(\inn{x}{\xi} - tq_2(\xi))} b^J_{k,k/3}(2^{-k}\xi,t) \hat{f}(\xi)\,\ud \xi
\end{equation*}
and
\begin{equation*}
    U^J_{k,\ell}f(x,t) := \sum_{\pm} \int_{\widehat{\R}^3} e^{i(\inn{x}{\xi} - tq_1^{\pm}(\xi))} b^{J,\pm}_{k,\ell}(2^{-k}\xi,t) \hat{f}(\xi)\,\ud \xi
\end{equation*}
for $\ell \in \Lambda(k)$. Let $\bJ \in \fJ^{3, \mathrm{sep}}(\delta)$ and $\ell_{\bJ} = (\ell_J)_{J \in \bJ}$ with $\ell_J \in \overline{\Lambda}(k)$. In light of Lemma~\ref{lem: stationary phase}, to prove Proposition~\ref{prop:L2 L3 trilinear}, it suffices to show 
\begin{equation}\label{eq: trilinear red 1}
   \Big\| \prod_{J \in \bJ}|U_J f_J|^{1/3}\Big\|_{L^3(B^{3+1}(0,1))}\lesssim_{\varepsilon} \delta^{-O(1)} 2^{k/6 - |\ell_{\bJ}|/9 + \varepsilon k}\prod_{J \in \bJ}\|f_J\|_{L^2(\R^3)}^{1/3}
\end{equation}
for $U_J := U^J_{k,\ell_J}$ and $f_J \in \Sc(\R^3)$ for $J \in \bJ$. This reduction follows from a standard localisation argument since the kernels
$K_{k,\ell}^J$ associated to the propagators $U_{k,\ell}^J$ satisfy the bounds
\begin{equation*}
    |K_{k,\ell}^J(x,t)| \lesssim_N 2^{k(d-N)} |x|^{-N} \quad \text{ for all $|x| \gtrsim 1$, $N>0$,}
\end{equation*}
via an integration-by-parts argument.

We may remove the $t$-dependence from the symbols $b^J_{k,k/3}$ and $b^{J,\pm}_{k,\ell}$ using a standard Fourier series expansion argument. Owing to the $L^2$-norm on the right-hand side of \eqref{eq: trilinear red 1}, we may also freely exchange $f$ and $\hat{f}$. Thus, after rescaling, we are led to consider operators of the form
\begin{equation*}
    T^J_{k, k/3}g(x,t) := \int_{\widehat{\R}^3} e^{i(\inn{x}{\xi} - tq_2(\xi))} b^J_{k,k/3}(\xi) g(\xi)\,\ud \xi
\end{equation*}
and
\begin{equation*}
    T^J_{k, \ell}g(x,t) := \sum_{\pm} \int_{\widehat{\R}^3} e^{i(\inn{x}{\xi} - tq_1^{\pm}(\xi))} b^{J,\pm}_{k,\ell}(\xi) g(\xi)\,\ud \xi 
\end{equation*}
for $\ell \in \Lambda(k)$, where the symbols $b^J_k$, $b^{J,\pm}_{k,\ell} \in C^{\infty}(\widehat{\R}^3)$ are bounded in absolute value by $1$ and further satisfy
\begin{equation}\label{eq: trilinear red 2a}
    \supp b^J_{k,k/3} \subseteq \big\{\xi \in B^3(0,10) : \mathrm{dist}(\theta_2(\xi), J) < \delta \textrm{ and } |u(\xi)| \lesssim 2^{-2k/3 + 4 \varepsilon k /3} \big\}
\end{equation}
and
\begin{equation}\label{eq: trilinear red 2b}
    \supp b^{J,\pm}_{k,\ell} \subseteq \big\{\xi \in B^3(0,10) : \mathrm{dist}(\theta_2(\xi), J) < \delta \textrm{ and } |u(\xi)| \sim 2^{-2\ell} \big\}
\end{equation}
for $\ell \in \Lambda(k)$. In particular, to prove the desired estimate \eqref{eq: trilinear red 1}, it suffices to show 
\begin{equation}\label{eq: trilinear red 3}
   \Big\| \prod_{J \in \bJ} |T_J g_J|^{1/3}\Big\|_{L^3(B^{3+1}(0,2^k))}\lesssim_{\varepsilon} \delta^{-O(1)} 2^{- |\ell_{\bJ}|/9 + \varepsilon k}\prod_{J \in \bJ}\|g_J\|_{L^2(\R^3)}^{1/3}
\end{equation}
for $T_J := T^J_{k,\ell_J}$  and $g_J \in \Sc(\R^3)$ for $J \in \bJ$ and $\ell_{\bJ} = (\ell_J)_{J \in \bJ}$ with $\ell_J \in \overline{\Lambda}(k)$.




\subsection{Verifying the hypotheses of Theorem~\ref{thm: local BCT}} Enumerate the intervals in $\bJ$ as $J_1$, $J_2$, $J_3$ so that, writing $\ell_i := \ell_{J_i}$ and $T_i := T_{J_i}$, we have $0 < \ell_1 \leq \ell_2 \leq \ell_3 \leq k/3$, $\ell_i \in \overline{\Lambda}(k)$. The trilinear estimate \eqref{eq: trilinear red 3} is a consequence of Theorem~\ref{thm: local BCT} where we exploit the additional localisation of the symbol of $T_3$ to the set $|u(\xi)| \lesssim 2^{-2\ell_3}$.\footnote{Or the slightly larger set $|u(\xi)| \lesssim 2^{-2k/3 + 4 \varepsilon k /3}$ in the case $\ell_3 = k/3$.} In order to apply Theorem~\ref{thm: local BCT}, we must verify the regularity and transversality hypotheses. 

We begin by noting, as a consequence of the definition of the functions $\theta_2$ and $\theta_1^{\pm}$, that
\begin{equation}\label{eq: trilinear red 4}
    \nabla q_2(\xi) = \gamma \circ \theta_2(\xi) + u(\xi)\nabla \theta_2(\xi) \quad \textrm{and} \quad \nabla q_1^{\pm}(\xi) = \gamma \circ \theta_1^{\pm}(\xi).
\end{equation}

\subsubsection*{Regularity hypothesis} We first show that the functions $q_2$ and $q_1^{\pm}$ all satisfy (at least) the $C^{1, 1/2}$ condition. It is easy to see that the function $q_2$ is $C^{1,1}$ in the sense that
\begin{equation*}
    |\nabla q_2(\xi_1) - \nabla q_2(\xi_2)| \lesssim |\xi_1 - \xi_2| \qquad \textrm{for all  $\xi_1$, $\xi_2 \in \supp b_{k,k/3}^J$.}
\end{equation*}
On the other hand, the functions $q_1^{\pm}$ are less regular and only satisfy a $C^{1,1/2}$ condition, as first observed in~\cite[Lemma 3.5]{KLO2022}.

\begin{lemma}\label{lem: C 1 1/2} For $\ell \in \Lambda(k)$, we have
    \begin{equation*}
    |\nabla q_1^{\pm}(\xi_1) - \nabla q_1^{\pm}(\xi_2)| \lesssim \min\big\{2^{-\ell}, 2^{\ell}|\xi_1 - \xi_2| \big\} + |\xi_1 - \xi_2|  \lesssim |\xi_1 - \xi_2|^{1/2} 
\end{equation*}
for all $\xi_1$, $\xi_2 \in \supp b_{k,\ell}^{J,\pm}$.
\end{lemma}

\begin{proof} Fix $\xi_1$, $\xi_2 \in \supp b_{k,\ell}^{J,\pm}$. In view of \eqref{eq: trilinear red 4}, it suffices to show
\begin{equation}\label{eq: reg 1}
    |\theta_1^{\pm}(\xi_1) - \theta_1^{\pm}(\xi_2)| \lesssim |\xi_1 - \xi_2|^{1/2}.
\end{equation}
By differentiating the defining function for $\theta_1^{\pm}$, we see that
\begin{equation*}
    |\nabla \theta_1^{\pm}(\xi)| = \frac{|\gamma' \circ \theta_1^{\pm}(\xi)|}{|v^{\pm}(\xi)|} \sim 2^{\ell} \qquad \textrm{for $\xi \in \supp b_{k,\ell}^{J,\pm}$.}
\end{equation*}
Thus, the fundamental theorem of calculus implies\footnote{Here we must be a little careful in applying the fundamental theorem of calculus because the crucial condition $|u(\xi)| \sim 2^{-2\ell}$ does not hold on a convex set. If $|\xi_1 - \xi_2| \ll 2^{-2\ell}$, then this presents no problem. However, if $|\xi_1 - \xi_2| \gtrsim 2^{-2 \ell}$, then to apply the the fundamental theorem of calculus we construct a continuous, piecewise smooth curve connecting $\xi_1$ to $\xi_2$ which consists of two linear segments of length $O(2^{-2\ell})$ and a curve of length $O(|\xi_1- \xi_2|)$ which traverses the level set $\{\xi \in B(0,10) : u(\xi) = 2^{-2\ell}\}$.}
\begin{equation}\label{eq: reg 2}
     |\theta_1^{\pm}(\xi_1) - \theta_1^{\pm}(\xi_2)| \lesssim 2^{\ell}|\xi_1 - \xi_2|. 
\end{equation}
On the other hand, by Lemma~\ref{lem: root},
\begin{align}
\nonumber
    |\theta_1^\pm(\xi_1) - \theta_1^\pm(\xi_2)| &\leq |\theta_1^\pm(\xi_1) - \theta_2(\xi_1)| + |\theta_2(\xi_1) - \theta_2(\xi_2)|+ |\theta_2(\xi_2) - \theta_1^\pm(\xi_2)| \\
    \label{eq: reg 3}
    &\lesssim |\xi_1-\xi_2| + 2^{-\ell}.
\end{align}

Combining \eqref{eq: reg 2} and \eqref{eq: reg 3}, we deduce that 
\begin{equation*}
   |\theta_1^{\pm}(\xi_1) - \theta_1^{\pm}(\xi_2)| \lesssim \min\big\{2^{-\ell}, 2^{\ell}|\xi_1 - \xi_2| \big\}  + |\xi_1 - \xi_2|
\end{equation*}
which is precisely the first inequality in \eqref{eq: reg 1}. Furthermore,
\begin{equation*}
 \min\big\{2^{-\ell}, 2^{\ell}|\xi_1 - \xi_2| \big\} = \min\big\{2^{-\ell}|\xi_1 - \xi_2|^{-1/2}, 2^{\ell}|\xi_1 - \xi_2|^{1/2} \big\}  |\xi_1 - \xi_2|^{1/2} \leq |\xi_1 - \xi_2|^{1/2},  
\end{equation*}
and, since $|\xi_j| \lesssim 1$ for $\xi_j \in \supp b_{k,\ell}^{J,\pm}$, the second inequality in \eqref{eq: reg 1} immediately follows.
\end{proof}

\subsubsection*{Transversality hypothesis} We now turn to verify the transversality hypothesis from \eqref{eq: trans}. This involves estimating expressions of the form
\begin{equation*}
   \bigg| \det
    \begin{pmatrix}
        \nabla Q_1(\xi_1) &  \nabla Q_2(\xi_2) & \nabla Q_3(\xi_3) & \nabla u(\xi_3) \\
        1 & 1 & 1 & 0
    \end{pmatrix} 
    \bigg|
\end{equation*}
where each $Q_j$ is either of the functions $q_2$ or $q_1^{\pm}$. The columns where $Q_j = q_2$ are slightly more complicated since the formula for $\nabla q_2(\xi)$ in \eqref{eq: trilinear red 4} involves multiple terms. However, we can always treat the second term as an error and effectively ignore it. Indeed, if $Q_j = q_2$, then we must have $\ell_j = k/3$ and so we consider $\xi \in \supp b_{k,k/3}^J$. In this case, $|u(\xi)| \lesssim 2^{-2k/3 + 4\varepsilon k/3}$ and, by differentiating the defining equation for $\theta_2$, we also have $|\nabla \theta_2(\xi)| \lesssim 1$. Since $k$ is large, we can therefore think of $\nabla q_2(\xi)$ as a tiny perturbation of $\gamma \circ \theta_2(\xi)$ on $\supp b_{k,k/3}^J$. On the other hand, for the final column we have
\begin{equation}\label{eq: trilinear red 5}
     \nabla u(\xi) = \gamma' \circ \theta_2(\xi). 
\end{equation}

In view of the support conditions \eqref{eq: trilinear red 2a} and \eqref{eq: trilinear red 2b} and the derivative formul\ae\ \eqref{eq: trilinear red 4} and \eqref{eq: trilinear red 5}, the transversality hypothesis would follow from the bound
\begin{equation}\label{eq: trilinear red 7}
  \left| \det \begin{pmatrix}
        \gamma(s_1) & \gamma(s_2) & \gamma(s_3) & \gamma'(s_3) \\
        1 & 1 & 1 & 0
    \end{pmatrix} \right|
    \gtrsim |V(s_1, s_2, s_3)||s_1-s_3||s_2-s_3|
\end{equation}
for all $s_1, s_2, s_3 \in I$, where here and below
\begin{equation*}
    V(t_1, \dots, t_m) := \prod_{1 \leq i < j \leq m} (t_i - t_j)
\end{equation*}
denotes the Vandermonde determinant in the variables $t_1, \dots, t_m \in \R$. In order to make this reduction, we use the bound $|\theta_2(\xi) - \theta_1^{\pm}(\xi)| \lesssim 2^{-\ell} < \delta^{8}$ from Lemma~\ref{lem: root}. For $\ell_3 \in \Lambda(k)$, we can think of $\theta_2(\xi)$ and $\theta_1^{\pm}(\xi)$ as approximately equal; this allows us to reduce to a situation involving only three variables $s_1$, $s_2$, $s_3$. Here we use the fact that the right-hand side of \eqref{eq: trilinear red 7} is bounded below by (a constant multiple of) $\delta^5$ for $s_i \in I$ with $\dist(s_i, J_i) < 2\delta$ for $ 1\leq i \leq 3$ and $ \mathbf{J}=(J_1,J_2,J_3) \in \mathfrak{J}^{3,\mathrm{sep}}(\delta)$. 

 By repeated application of the fundamental theorem of calculus, we may express the left-hand determinant in \eqref{eq: trilinear red 7} as
\begin{equation*}
    \int_{s_1}^{s_2}
    \int_{s_2}^{s_3}
    \int_{t_2}^{t_3}
    \int_{t_3}^{s_3}
    \int_{u_3}^{u_4}
    \det \begin{pmatrix}
        \gamma(s_1) & \gamma'(t_2) & \gamma''(u_3) & \gamma'''(v_4) \\
        1 & 0 & 0 & 0
    \end{pmatrix} 
    \,\ud v_4 \ud u_4 \ud u_3 \ud t_3 \ud t_2. 
\end{equation*}
By continuity and the reductions in \S\ref{subsec: pert moment}, the inner determinant is single-signed and bounded below in absolute value by some constant.

By the observations of the previous paragraph, the left-hand side of \eqref{eq: trilinear red 7} is comparable to the same expression but with $\gamma$ replaced by the moment curve $\gamma_{\circ}(s) := (s, s^2/2, s^3/6)$. Consequently, it suffices to prove \eqref{eq: trilinear red 7} for this particular curve. However, in this case the left-hand side of \eqref{eq: trilinear red 7} corresponds to (the absolute value of a scalar multiple of) $\partial_t V(s_1, s_2, s_3, s_3 + t)|_{t=0}$. A simple calculus exercise shows this agrees with the expression appearing in the right-hand side of \eqref{eq: trilinear red 7}, as required. For similar arguments, see~\cite{GGPRY2021, KLO2022}.\medskip




\section{\texorpdfstring{$L^{12}$}{4}-local smoothing}\label{sec: L12 loc smoothing}

In this section we upgrade the $L^p$-local smoothing estimates of~\cite{PS2007} for $p \geq 12$ by exploiting the localisation of the spatio-temporal Fourier transform of $m[a_{k,\ell}](D;t)f$ with respect to the $2$-dimensional cone $\Gamma_2$ in $(\xi, \tau)$-space from the introduction. The arguments are similar to those of \cite{KLO2023} and \cite{BGHS-Sobolev}. Crucially, we apply a decoupling inequality from~\cite{BGHS-Sobolev}, which is a conic variant of the celebrated decoupling inequality for non-degenerate curves~\cite{BDG2016, GLYZ}. After this step, the remainder of the argument is similar to that of~\cite{PS2007}.\footnote{Decoupling is also used in~\cite{PS2007}, but only with respect to a cone in $\xi$-space, leading to non-sharp regularity estimates.}




\subsection{Decomposition along the curve} 

Fix $\zeta \in C^{\infty}(\R)$ with $\supp \zeta \subseteq I$ such that $\sum_{\mu \in \Z} \zeta(\,\cdot\, - \mu) \equiv 1$. 
For $k \in \N$ and $\ell \in \overline{\Lambda}(k)$, we write
\begin{equation*}
    a_{k,\ell}(\xi;t;s) =  \sum_{\mu \in \Z}  a_{k,\ell}(\xi;t;s)\zeta(2^{\ell}\theta_2(\xi) - \mu).
\end{equation*}
Using stationary phase arguments, as in the proof of Lemma~\ref{lem: stationary phase}, we may localise $s$ to lie in a neighbourhood of $\theta_2(\xi)$. Let $\rho>0$ be a small `fine tuning' constant chosen to satisfy the requirements of the forthcoming argument; for instance, we may take $\rho := 10^{-6}$. We then define
\begin{align*}
    a_{k,\ell}^{\mu}(\xi;t;s) &:= 
       a_{k,\ell}(\xi;t;s)\zeta(2^{\ell}\theta_2(\xi) - \mu)\eta\big(\rho 2^{\ell}(s-\theta_2(\xi))\big) \quad \textrm{for $\ell \in \Lambda(k)$,} \\ \nonumber
    a_{k,k/3}^{\mu}(\xi;t;s) &:= a_{k,\ell}(\xi;t;s)\zeta(2^{k/3}\theta_2(\xi) - \mu)\eta\big(\rho 2^{k/3-\varepsilon k}(s-\theta_2(\xi))\big)
\end{align*}
for all $\mu \in \Z$.

\begin{lemma}\label{J=3 s loc lem} Let $2 \leq p < \infty$.
For all $k \in \N$ and $\ell \in \overline{\Lambda}(k)$, we have
    \begin{equation*}\label{eq: J=3 s loc lem}
        \Big\|m\Big[a_{k,\ell} - \sum_{\mu \in \Z} a_{k,\ell}^{\mu}\Big](D; \,\cdot\,) f \Big\|_{L^p(\R^{3+1})}  \lesssim_{N,\varepsilon,p} 2^{-kN} \| f \|_{L^p(\R^3)} \qquad \textrm{for all $N \in \N$.}
    \end{equation*}
\end{lemma}

By Lemma~\ref{lem: root}, if $|s - \theta_2(\xi)| \gg 2^{-\ell}$, then $s$ is far from any roots $\theta_1^{\pm}(\xi)$ of the phase function. Hence Lemma~\ref{eq: J=3 s loc lem} follows from non-stationary phase, as in the analysis of the error term in Lemma~\ref{lem: stationary phase} ii). Moreover, the proof is very similar to that of~\cite[Lemma 8.1]{BGHS-helical} and we therefore omit the precise details.\footnote{The argument is in fact entirely the same as the proof of the case $\floor{k/3}_\varepsilon \leq \ell \leq \floor{k/3}$ from~\cite[Lemma 8.1]{BGHS-helical}.}

The support properties of the symbols $a_{k,\ell}^\mu$ are best understood in terms of the \textit{Frenet frame}. Recall, given a smooth non-degenerate curve $\gamma:I \to \R^d$, the Frenet frame is the orthonormal basis $\{\be_1(s), \dots, \be_d(s)\}$ resulting from applying the Gram--Schmidt process to  the vectors $\{ \gamma'(s), \dots, \gamma^{(d)}(s)\}$. With this setup, given $0 < r \leq 1$ and $s \in I$, recall the definition of the $(1,r)$-\textit{Frenet boxes} $\pi_{1}(s;\,r)$ introduced in~\cite{BGHS-helical}; namely,
\begin{equation*}
    \pi_1(s;\,r):= \big\{ \xi \in \widehat{\R}^3: |\inn{\be_j(s)}{\xi}| \lesssim r^{3-j} \,\, \textrm{for $j=1,\,2$}, \quad 
    |\inn{\be_{3}(s)}{\xi}| \sim 1\big\}. 
\end{equation*}
We  have the following support property.

\begin{lemma}\label{J=3 3d supp lem}
Let $k \in \N$, $\ell \in \overline{\Lambda}(k)$ and $\mu\in \Z$. Then
\begin{equation*}
   \xisupp a_{k,\ell}^{\mu} \subseteq 2^k \cdot \pi_1(s_{\mu}; 2^{-\ell}). 
\end{equation*}
\end{lemma}

The proof is similar to that of~\cite[Lemma 8.2, (a)]{BGHS-helical}, so we omit the details.




\subsection{Spatio-temporal localisation}

The symbols are further localised with respect to the Fourier transform of the $t$-variable. In particular, define the homogeneous phase function $q_2(\xi) := \inn{\gamma\circ\theta_2(\xi)}{\xi}$  as in Lemma~\ref{lem: stationary phase} and let
\begin{equation*}
\chi_{k,\ell}(\xi, \tau) := \eta\big(2^{-(k-3\ell)-4\varepsilon k} (\tau + q_2(\xi))\big) \qquad \textrm{for $\ell \in \overline{\Lambda}(k)$.}
\end{equation*}
We introduce the localised multipliers $m_{k,\ell}^{\mu}$, defined by
\begin{equation*}
    \mathcal{F}_t \big[m_{k,\ell}^{\mu}(\xi;\,\cdot\,)\big](\tau):=    \chi_{k,\ell}(\xi, \tau) \, \mathcal{F}_t\big[ m[a_{k,\ell}^{\mu}](\xi; \,\cdot\,) \big](\tau).
\end{equation*}
 Here $\mathcal{F}_t$ denotes the Fourier transform acting in the $t$ variable. A stationary phase argument allows us to pass from $m[a_{k,\ell}^{\mu}]$ to $m_{k,\ell}^{\mu}$.

\begin{lemma}\label{J=3 spatio temp loc lem}
Let $2 \leq p < \infty$.
For all $k \in \N$, $\ell \in \overline{\Lambda}(k)$ and $\mu \in \Z$, we have
    \begin{equation*}
        \big\|\big(m[a_{k,\ell}^{\mu}] - m_{k,\ell}^{\mu}\big)(D;  \,\cdot\,) f \big\|_{L^p(\R^{3+1})}  \lesssim_{N,\varepsilon} 2^{-kN} \| f \|_{L^p(\R^3)} \qquad \textrm{for all $N \in \N$.}
    \end{equation*}
\end{lemma}

The proof, which is based on a fairly straightforward integration-by-parts argument, is similar to that of~\cite[Lemma 8.3]{BGHS-helical} and we omit the details. 

To understand the support properties of the multipliers $m_{k,\ell}^{\mu}$, we introduce the primitive curve
\begin{equation*}
   \bar{\gamma} \colon I \to \R^4, \qquad \bar{\gamma} \colon s \mapsto \begin{bmatrix}
    \int_0^s \gamma \\
    s
    \end{bmatrix}.
\end{equation*}
Here $\int_0^s \gamma$ denotes the vector in $\R^3$ with $i$th component $\int_0^s \gamma_i$ for $1 \leq i \leq 3$. Note that $\bar{\gamma}$ is a non-degenerate curve in $\R^4$ and, in particular, \begin{equation*}
    |\det
    \begin{pmatrix}
        \bar{\gamma}^{(1)}(s) & \cdots & \bar{\gamma}^{(4)}(s)
    \end{pmatrix}| = 
    |\det
    \begin{pmatrix}
        \gamma^{(1)}(s) & \cdots & \gamma^{(3)}(s)
    \end{pmatrix}
    | \qquad \textrm{for all $s \in I$.}
\end{equation*}
Let $(\bar{\be}_j(s))_{j=1}^4$ denote the Frenet frame associated to $\bar{\gamma}$ and consider the $(2,r)$-Frenet boxes for $\bar \gamma$
\begin{equation*}
    \pi_{2,\bar{\gamma}}(s;\,r) := \big\{\Xi = (\xi, \tau) \in \widehat{\R}^3 \times \widehat{\R} : |\inn{\bar{\be}_j(s)}{\Xi}| \lesssim r^{4 - j} \textrm{ for $1 \leq j \leq 3$, } |\inn{\bar{\be}_4(s)}{\Xi}| \sim 1 \big\},
\end{equation*}
as introduced in~\cite{BGHS-helical}. 

\begin{lemma}\label{J=3 4d supp lem} For all $k \in \N$, $\ell \in \overline{\Lambda}(k)$ with $\ell \geq \ceil{4 \varepsilon k}$ and $\mu \in \Z$, we have
    \begin{equation*}
        \supp \mathcal{F}_t \big[m_{k,\ell}^{\mu}\big] \subseteq 2^k \cdot \pi_{2,\bar{\gamma}}(s_{\mu}; 2^{4\varepsilon k}2^{-\ell}),
    \end{equation*} 
     where $s_{\mu}:=2^{-\ell }\mu$ and $\mathcal{F}_t$ denotes the Fourier transform in the $t$-variable.
\end{lemma}

The proof, which is based on a fairly straightforward integration-by-parts argument, is similar to that of~\cite[Lemma 8.4]{BGHS-helical} and we omit the details.




\subsection{A decoupling inequality}

With the above observations, we can immediately apply the decoupling inequalities in~\cite[Theorem 4.4]{BGHS-Sobolev} associated to the primitive curve $\bar \gamma$ to isolate the contributions from the individual $m_{k,\ell}^{\mu}(D; \,\cdot\,)$.

\begin{proposition}\label{J=3 decoupling}
Let $p \geq 12$. For all $k \in \N$, $\ell  \in \overline{\Lambda}(k)$, we have
\begin{equation*}
    \Big\|\sum_{\mu \in \Z} m_{k,\ell}^{\mu}(D;\, \cdot \,) f \Big\|_{L^p(\R^{3+1})} \lesssim_{\varepsilon}  2^{O(\varepsilon k)} 2^{\ell (1-\frac{7}{p})} \Big( \sum_{\mu \in \Z} \|m_{k,\ell}^{\mu}(D; \,\cdot\,) f\|_{L^p(\R^{3+1})}^p \Big)^{1/p}.
\end{equation*}
\end{proposition}

\begin{proof} If $\ell \in \overline{\Lambda}(k)$ satisfies $\ell \leq \ceil{4 \varepsilon k}$, then we bound the left-hand side trivially using the triangle and the Cauchy--Schwarz inequalities. For the case $\ell > \ceil{4 \varepsilon k}$, we partition the family of sets $\pi_{2,\bar \gamma}(s_\mu; 2^{4 \varepsilon k} 2^{-\ell})$ for $|\mu| \leq 2^{\ell}$ into $O(2^{4 \varepsilon k})$ subfamilies, each forming a $(2, 2^{4 \varepsilon k} 2^{-\ell})$-Frenet box decomposition in the language of~\cite[\S4]{BGHS-Sobolev}. In view of Lemma \ref{J=3 4d supp lem} and after a simple rescaling, the result now follows from~\cite[Theorem 4.4]{BGHS-Sobolev} applied with $d=3$, $n=4$ to the primitive curve $\bar \gamma$. 
\end{proof}




\subsection{Localising the input function}
The Fourier multipliers $m_{k,\ell}^{\mu}(D; t)$ induce a localisation on the input function $f$. We recall the setup from~\cite[\S8.6]{BGHS-helical}. Given $\ell  \in \overline{\Lambda}(k)$ and $m \in \Z$ define
\begin{equation*}
\Delta_{k, \ell}(m) := \big\{ \xi \in \widehat{\R}^3 : |\xi_2 - \xi_3 2^{-\ell}m| \leq C 2^{-\ell} \xi_3 \textrm{ and } C^{-1} 2^k \leq \xi_3 \leq C 2^k \big\},
\end{equation*}
where $C \geq 1$ is an absolute constant, chosen sufficiently large so that the following lemma holds. 

\begin{lemma}[{\cite[Lemma 8.8]{BGHS-helical}}]\label{sec supp lem} 
If $\mu \in \Z$, then there exists some $ m(\mu) \in \Z$ such that
\begin{equation*}
    2^k \cdot \pi_1(s_{\mu}; 2^{-\ell}) \subseteq \Delta_{k, \ell}\big(m(\mu)\big).
\end{equation*}
Furthermore, for each fixed $k$ and $\ell$, given $m \in \Z$ there are only $O(1)$ values of $\mu \in \Z$ such that $m = m(\mu)$.
\end{lemma}

For each $\mu \in \Z$ define the smooth cutoff function 
\begin{equation*}
    \chi_{k,\ell}^{\mu}(\xi_2, \xi_3) := \eta\big(C^{-1}(2^{\ell}\xi_2/\xi_3 - m(\mu))\big) \, \big( \eta(C^{-1}2^{-k}\xi_3) - \eta(2C^{-1}2^{-k} \xi_3)\big).
\end{equation*}
If $(\xi_1, \xi_2, \xi_3) \in \xisupp a_{k,\ell}^{\mu}$, then Lemmas \ref{J=3 3d supp lem} and \ref{sec supp lem} imply $\chi_{k,\ell}^{\mu}(\xi_2, \xi_3) = 1$. Thus, if we define the corresponding frequency projection
\begin{equation*}
    f^{\mu}_{k,\ell} := \chi_{k,\ell}^{\mu}(D)f,
\end{equation*}
it follows that
\begin{equation}\label{eq:localising input}
    m[a_{k,\ell}^{\mu}](D;\,\cdot\,)f = m[a_{k,\ell}^{\mu}](D;\,\cdot\,)f^{\mu}_{k,\ell}.
\end{equation}

\begin{lemma}\label{lemma:sum sectors}
For all $k \in \N$ and $\ell  \in \overline{\Lambda}(k)$, we have
    \begin{equation*}
        \Big( \sum_{\mu \in \Z} \|f_{k,\ell}^{\mu}\|_{L^{p}(\R^{3})}^{p} \Big)^{1/p} \lesssim \| f \|_{L^p(\R^3)}.
    \end{equation*}
\end{lemma}

\begin{proof}
    The case $p=2$ follows from Plancherel's theorem via Lemma \ref{sec supp lem} and the finite overlapping of the sectors $\Delta_{k,\ell}(m(\mu))$. For $p=\infty$, it is easy to see that $\sup_\mu \| \mathcal{F}_{\xi_2,\xi_3}^{-1} [\chi_{k,\ell}^\mu] \|_{L^1(\R^2)} \lesssim 1$; indeed this is immediate for $k = \ell = 0$ and the general case follows since $\chi_{k,\ell}^{\mu}(\xi_2, \xi_3) = \chi_{0,0}^{\mu}(2^{\ell-k}\xi_2, 2^{-k}\xi_3)$. Interpolating these two cases, using mixed-norm interpolation (see, for instance,~\cite[\S 5.6]{BL1976}) concludes the proof.
\end{proof}




\subsection{Local smoothing for the \texorpdfstring{$m[a_{k,\ell}^{\mu}](D,t)$}{3}}

We recall the following result, which follows from~\cite[Theorem 4.1]{PS2007} when combined with the main result from~\cite{BD2015}.

\begin{theorem}[{\cite[Theorem 4.1]{PS2007}}]\label{PS LS J=2}
  Let $\gamma \colon I \to \R^3$ be a smooth curve and suppose that $a \in C^{\infty}(\widehat{\R}^3\setminus \{0\} \times \R \times \R)$ satisfies the symbol conditions
\begin{equation*}
    |\partial_{\xi}^{\alpha}\partial_t^i \partial_s^j a(\xi;t;s)| \lesssim_{\alpha, i, j} |\xi|^{-|\alpha|} \qquad \textrm{for all $\alpha \in \N_0^3$ and $i$, $j \in \N_0$}
\end{equation*}
and that
\begin{equation*}
    |\inn{\gamma'(s)}{\xi}| + |\inn{\gamma''(s)}{\xi}| \gtrsim |\xi| \qquad \text{ for all $\,\, (\xi;s) \in \xisupp a \times I$}.
\end{equation*}
Let $ p \geq 6$, $\varepsilon>0$ and $k \geq 1$. If $a_k$ is defined as in \eqref{symbol dec}, then
\begin{equation*}
    \Big(\int_1^2\|m[a_k](D;t)f\|_{L^p(\R^3)}^p\,\ud t\Big)^{1/p} \lesssim_{\varepsilon, p} 2^{-2k/p + \varepsilon k}\|f\|_{L^p(\R^3)}.
\end{equation*}
\end{theorem}

By combining Theorem \ref{PS LS J=2} with the rescaling from \S\ref{subsec: scaling} we obtain the following bound for our multipliers $m[a_{k,\ell}^{\mu}]$.  

\begin{proposition}\label{prop:LS J=2}
Let $p \geq 6$. For all $k \in \N$, $\ell  \in \overline{\Lambda}(k)$ and $\mu \in \Z$ we have
\begin{equation}\label{eq:LS J=2}
    \Big( \int_1^2\| m[a_{k,\ell}^{\mu}](D; t) f \|_{L^p(\R^3)}^p \, \ud t \Big)^{1/p} \lesssim_{\varepsilon} 2^{O(\varepsilon k)} 2^{-2(k-3\ell)/p -\ell} \| f \|_{L^p(\R^3)}.
\end{equation}
\end{proposition}

\begin{proof} The argument is essentially the same proof as that of~\cite[Proposition 5.1]{PS2007}. We distinguish two cases: \medskip

    \noindent \underline{Case: $\ell = k/3$}. The result follows from interpolation between the elementary estimates
    \begin{align*}
        \sup_{t \in \R} \| m[a_{k,k/3}^{\mu}](\,\cdot\,; t) \|_{M^2(\R^3)}  &\lesssim 2^{-k/3 +O(\varepsilon k)}, \\ \sup_{t \in \R}\| m[a_{k,k/3}^{\mu}](\,\cdot\,; t)  \|_{M^\infty(\R^3)}  &\lesssim 2^{-k/3 +O(\varepsilon k)}.
    \end{align*}
    Both these inequalities are consequences of the size of the $s$ support of $a_{k,k/3}^{\mu}$. The first is trivial. The second can be deduced, for instance, by adapting arguments from~\cite[\S 5.6]{BGHS-Sobolev}.
  \medskip

    \noindent \underline{Case: $\ell  \in \Lambda(k)$}. Fix $\ell \in \Lambda(k)$, $\mu \in \Z$ and set $\sigma := s_{\mu}$ and $\lambda := 2^{-\ell}$. With the notation from \S\ref{subsec: scaling}, we define 
    \begin{equation*}
   \tilde{\gamma}:= \gamma_{\sigma, \lambda}, \quad \tilde{a} := (a_{k,\ell}^{\mu} )_{\sigma, \lambda}, \quad \tilde{f}:= f_{\sigma, \lambda}, \quad \tilde{M} := ([\gamma]_{\sigma, \lambda})^{-\top}. 
    \end{equation*}
Thus, in view of \eqref{eq:scaling mult identity}, we have
      \begin{equation}\label{eq:scaling mult identity tilda}
        m_{\gamma}[a_{k,\ell}^{\mu}](D; t) f(x) = \lambda \big(m_{\tilde{\gamma}}[\tilde{a}](D; t) \tilde{f}\big)\big(\tilde{M}^{\top} x - t\gamma(\sigma) \big).
    \end{equation}

We claim $\tilde{\gamma}$ and $\tilde{a}$ satisfy the hypotheses of Theorem~\ref{PS LS J=2}.  By \eqref{eq: resc symb supp}, we have 
 \begin{equation}\label{eq: resc supp cond}
    |\xi| \sim 2^{k-3\ell} \qquad \textrm{for all $\xi \in \xisupp \tilde{a}$.}  
 \end{equation}
Combining this with \eqref{eq:inner prods scaled} and \eqref{eq: weak non-deg}, we have
    \begin{equation}\label{eq:2-non-degenerate rescaled}
        |\inn{\tilde{\gamma}'(s)}{\xi}| + |\inn{\tilde{\gamma}''(s)}{\xi}| \gtrsim |\xi| \qquad \text{for all $(\xi; t; s) \in \supp \, \tilde{a}$.}
    \end{equation} 
On the other hand, let $\tilde{\theta}_2$ and $\tilde{u}$ be the functions defined in \S\ref{subsec: microloc dec}, but now with respect to the curve $\tilde{\gamma}$. It follows that
\begin{equation*}
    \sigma + \lambda \tilde{\theta}_2(\xi) = \theta_2\circ\tilde{M} (\xi) \quad \textrm{and} \quad \tilde{u}(\xi) = \lambda u \circ \tilde{M}(\xi)
\end{equation*}
and so 
\begin{equation*}
    \tilde{a}(\xi; t; s) = (a_k)_{\sigma, \lambda}(\xi; t; s) \beta\big(2^{-k+ 3\ell}\tilde{u}(\xi)\big)\zeta(\tilde{\theta}_2(\xi) - \sigma - \mu)\eta\big(\rho(s-\tilde{\theta}_2(\xi))\big).
\end{equation*}
Using the fact that $|\inn{\tilde{\gamma}^{(3)}\circ \tilde{\theta}_2(\xi)}{\xi}| \sim 2^{k-3\ell} \sim |\xi|$ on the support of $\tilde{a}$, it is a straightforward exercise to show that 
\begin{equation*}
    |\partial_{\xi}^{\alpha} \tilde{\theta}_2(\xi)| \lesssim_{\alpha} |\xi|^{-|\alpha|} \qquad \textrm{and} \qquad 2^{-k+ 3\ell}|\partial_{\xi}^{\alpha} \tilde{u}(\xi)| \lesssim_{\alpha} |\xi|^{-|\alpha|}
\end{equation*}
hold for all $\xi \in \xisupp \tilde{a}$ and $\alpha \in \N_0^3$. The derivative bounds
    \begin{equation}\label{eq:symbol bounds rescaled}
        | \partial_{\xi}^{\alpha} \partial_t^i \partial_{s}^j \tilde{a}(\xi; t; s)| \lesssim_{\alpha, i, j} |\xi|^{-|\alpha|} \quad \text{ for all $\alpha \in \N_0^3$ and $i, j \in \N_0$}
    \end{equation}
then easily follow, noting that the derivatives of $(a_k)_{\sigma, \lambda}$ can be controlled following the discussion at the end of \S\ref{subsec: scaling}.

     As a consequence of \eqref{eq: resc supp cond}, we may write $\tilde{a} = \sum_{j=0}^{\infty} \tilde{a}_j$ where each $\tilde{a}_j$ is a localised symbol as defined in \eqref{symbol dec} and the only non-zero terms of the sum correspond to values of $j$ satisfying $2^j \sim 2^{k-3\ell}$.  In view of \eqref{eq:2-non-degenerate rescaled} and \eqref{eq:symbol bounds rescaled}, for $p \geq 6$  we can apply Theorem \ref{PS LS J=2} to obtain
    \begin{equation*}
    \Big( \int_1^2\| m_{\tilde{\gamma}}[\tilde{a}](D; t) \tilde{f} \|_{L^p(\R^3)}^p \, \ud t \Big)^{1/p} \lesssim_{\varepsilon} 2^{O(\varepsilon k)} 2^{-2(k-3\ell)/p} \| \tilde{f} \|_{L^p(\R^3)}.
\end{equation*}
This, together with \eqref{eq:scaling mult identity tilda} and an affine transformation in the spatial variables, gives the desired inequality \eqref{eq:LS J=2}.
\end{proof}




\subsection{Putting everything together} With the above ingredients, we can now conclude the proof of the $L^{12}$ local smoothing estimate.

\begin{proof}[Proof of Proposition \ref{prop:L12}]
By successively applying Lemma \ref{J=3 s loc lem}, Lemma \ref{J=3 spatio temp loc lem}, Proposition \ref{J=3 decoupling} and a second application of Lemma \ref{J=3 spatio temp loc lem}, we obtain
    \begin{multline*}
    \| m[a_{k,\ell}](D; \cdot) f \|_{L^{12} (\R^{3+1})} \lesssim_{\varepsilon, N} 2^{\varepsilon k/2} 2^{ 5\ell/12 } \Big( \sum_{\mu \in \Z} \|m[a_{k,\ell}^{\mu}](D; \,\cdot\,) f\|_{L^{12}(\R^{3+1})}^{12} \Big)^{1/12} \\ + 2^{-kN} \| f \|_{L^{12}(\R^3)}
    \end{multline*}
    for any $N > 0$. By the localisation \eqref{eq:localising input} and Proposition~\ref{prop:LS J=2} we have
    \begin{equation*}
        \Big( \sum_{\mu \in \Z} \|m[a_{k,\ell}^{\mu}](D; \,\cdot\,) f\|_{L^{12}(\R^{3+1})}^{12} \Big)^{1/12} \lesssim 2^{\varepsilon k / 2} 2^{-(k-3\ell)/6 - \ell} \Big( \sum_{\mu \in \Z} \|f_{k,\ell}^{\mu}\|_{L^{12}(\R^{3})}^{12} \Big)^{1/12}.
    \end{equation*}
    Combining the above observations with an application of Lemma \ref{lemma:sum sectors} concludes the proof.
\end{proof}




\section{The non-degenerate case}\label{sec: non-deg}

In the non-degenerate case $\ell = 0$ we appeal to the classical (linear) Stein--Tomas restriction estimate, rather than the trilinear theory from \S\ref{sec: trilinear}.

\begin{proposition}\label{prop: ST}
  Let $\gamma:I \to \R^3$ be a smooth curve and suppose that $a \in C^{\infty}(\widehat{\R}^3\setminus \{0\} \times \R \times \R)$ satisfies the symbol conditions
\begin{equation*}
    |\partial_{\xi}^{\alpha}\partial_t^i \partial_s^j a(\xi;t;s)| \lesssim_{\alpha, i, j} |\xi|^{-|\alpha|} \qquad \textrm{for all $\alpha \in \N_0^3$ and $i$, $j \in \N_0$}
\end{equation*}
and that
\begin{equation}\label{eq: non-deg J=2}
    |\inn{\gamma'(s)}{\xi}| + |\inn{\gamma''(s)}{\xi}| \gtrsim |\xi| \qquad \text{ for all $\,\, (\xi;s) \in \xisupp a \times I$}.
\end{equation}
Let $k \geq 1$. If $a_k$ is defined as in \eqref{symbol dec}, then
\begin{equation*}
    \Big(\int_1^2\|m[a_k](D;t)f\|_{L^6(\R^3)}^6\,\ud t\Big)^{1/6} \lesssim  2^{k/3}\|f\|_{L^2(\R^3)}.
\end{equation*}
\end{proposition}

\begin{proof} 
Decomposing the symbol $a$ into sufficiently many pieces with small $\xi$ and $s$ support, the non-degeneracy condition \eqref{eq: non-deg J=2} can be strengthened to the following: there exists $B >1$ such that
\begin{equation}\label{eq: strengthened 1}
   B^{-1} |\xi| \leq |\inn{\gamma''(s)}{\xi}| \leq B |\xi| \quad \text{ for all $(\xi;s) \in \xisupp a \times I$}
\end{equation}
and there exists $s_* \in I$ such that
\begin{equation}\label{eq: strengthened 2}
    |\inn{\gamma'(s_*)}{\xi}| \leq 10^{-10} B |\xi| \quad \text{ for all $\xi \in \xisupp a$};
\end{equation}
see, for instance,~\cite[\S 4]{PS2007} or~\cite[Chapter 2]{GovindanSheri2023} for details of this type of reduction, which relies on the fact that the oscillatory integral $m[a_k]$ is rapidly decreasing if the phase function $s \mapsto \inn{\gamma(s)}{\xi}$ has no critical points. Under conditions \eqref{eq: strengthened 1} and \eqref{eq: strengthened 2}, there exists a unique smooth mapping $\theta_1 \colon \xisupp a \to I$ such that
\begin{equation}\label{eq: theta1 defining}
    \inn{\gamma' \circ \theta_1(\xi)}{\xi}=0 \quad \text{ for all $\xi \in \xisupp a$}.
\end{equation}
Let $q_1(\xi):=\inn{\gamma \circ \theta_1 (\xi)}{\xi}$. Arguing as in the proof of Lemma \ref{lem: stationary phase}, we may use \eqref{eq: strengthened 1} and van der Corput's lemma with second-order derivatives to  write
\begin{equation*}
    m[a_{k}](\xi;t) = 2^{-k/2} e^{-itq_1(\xi)} b_{k}(2^{-k}\xi;t) (\xi;t)
\end{equation*}
where $b_{k} \in C^{\infty}(\R^{3+1})$  is supported in $B(0,10)$ and satisfies 
\begin{equation*}
    |\partial_t^N b_{k}(2^{-k} \xi;t)| \lesssim_N 1 \qquad \textrm{for all $(\xi; t) \in \R^{3+1}$ and all $N \in \N_0$.}
\end{equation*}
Following the reductions of \S\ref{subsec: trilinear rest red}, we consider an operator $T$ of the form 
\begin{equation*}
    Tg(x,t) :=  \int_{B^3(0,1)} e^{i(\inn{x}{\xi} - tq_1(\xi))} b(\xi) g(\xi)\,\ud \xi
\end{equation*}
for $b \in C^{\infty}(\widehat{\R}^3)$ bounded in absolute value by $1$ which, by \eqref{eq: strengthened 1}, satisfies
\begin{equation*}
\supp b \subseteq \{ \xi \in B^3(0,1) : |v (\xi)| \sim 1\} \qquad \textrm{where} \qquad v(\xi) := \inn{\gamma''\circ \theta_1(\xi)}{\xi}.
\end{equation*}
In particular, to prove the lemma, with the above setup it suffices to show
\begin{equation}\label{eq: Stein--Tomas 1}
    \|T g\|_{L^6(B^{3+1}(0, 2^k))} \lesssim \|g\|_{L^2(\R^3)}.
\end{equation}

The inequality \eqref{eq: Stein--Tomas 1} follows from a generalisation of the classical Stein--Tomas restriction theorem due to Greenleaf~\cite{Greenleaf1981} (see also~\cite[Chapter VIII, \S5 C.]{Stein1993}). To apply this result, we need to show that $q_1$ is smooth over the support of $b$ and satisfies certain curvature conditions. 

Arguing as in the proof of Lemma~\ref{lem: C 1 1/2}, we see that $|\nabla \theta_1(\xi)| \sim 1$ on $\supp b$ and, furthermore, the function $q_1$ is easily seen to be smooth with bounded derivatives over this set. A simple computation shows that the hessian $\partial_{\xi \xi}^2 q_1(\xi)$ is the rank 1 matrix formed by the outer product of the vectors $\nabla \theta_1(\xi)$ and $\gamma' \circ \theta_1(\xi)$. By elementary properties of rank 1 matrices, $\partial_{\xi \xi}^2 q_1(\xi)$ therefore has a unique non-zero eigenvalue given by
\begin{equation*}
    \kappa(\xi) := \inn{\gamma' \circ \theta_1(\xi)}{\nabla \theta_1(\xi)}.
\end{equation*}

We claim that 
\begin{equation}\label{eq: Stein--Tomas 2}
    |\kappa(\xi)| \sim 1 \qquad \textrm{for all $\xi \in \supp b$;}
\end{equation}
geometrically, this means that the surface formed by taking the graph of $q_1$ over some open neighbourhood of $\supp b$ has precisely one non-vanishing principal curvature. This is precisely the geometric condition needed to apply the result of~\cite{Greenleaf1981} in order to deduce \eqref{eq: Stein--Tomas 1}. To see \eqref{eq: Stein--Tomas 2} holds, we take the $\xi$-gradient of the defining equation \eqref{eq: theta1 defining} for $\theta_1$ and then form the inner product with $\nabla \theta_1(\xi)$ to deduce that
\begin{equation*}
    0 = \inn{\gamma' \circ \theta_1(\xi)}{\nabla \theta_1(\xi)} + \inn{\gamma'' \circ \theta_1(\xi)}{\xi}|\nabla \theta_1(\xi)|^2 = \kappa(\xi) + v(\xi)|\nabla \theta_1(\xi)|^2.
\end{equation*}
Since $|v(\xi)| \sim |\nabla \theta_1(\xi)| \sim 1$ on $\supp b$, the claim follows. 
\end{proof}

One can interpolate Proposition \ref{prop: ST} 
with the diagonal $L^6 \to L^6$ local smoothing  result of Pramanik--Seeger~\cite{PS2007} (see Theorem~\ref{PS LS J=2} above) to directly deduce the desired $L^4 \to L^6$ estimate for the non-degenerate piece $a_{k,0}$ introduced in \eqref{eq: ak0 def}.

\begin{lemma}\label{lem: non-deg loc smoothing} For all $k \in \N$ and $\varepsilon > 0$, we have
    \begin{equation*}
    \Big(\int_1^2\|m[a_{k,0}](D, t)f\|_{L^6(\R^3)}^6\,\ud t\Big)^{1/6} \lesssim_{\varepsilon} \delta^{-O(1)} 2^{-k/6 + \varepsilon k} \|f\|_{L^4(\R^3)}.
\end{equation*}
\end{lemma}

This lemma reduces the proof of Proposition~\ref{prop:L4 to L6} to establishing the $L^4 \to L^6$ bound in \eqref{eq: critical est} with $a_k$ replaced with the localised symbol $\fa_k := a_k-a_{k,0}$.




\section{Concluding the argument}\label{sec: trilin to lin}

Here we conclude the proof of Proposition~\ref{prop:L4 to L6} and, in particular, present the details of the trilinear reduction discussed in \S\ref{subsec: trilinear red}. 

\begin{proof}[Proof (of Proposition~\ref{prop:L4 to L6})] Fix $\varepsilon > 0$ and let $0 < \delta < 1$, $M_{\varepsilon} \in \N$ and $\bC_{\varepsilon} \geq 1$ be constants, depending only on $\varepsilon$, and chosen to satisfy the forthcoming requirements of the proof. We proceed by inducting on the parameter $k$. For $0 \leq k \leq \delta^{-100}$, the result is trivial and this serves as the base case. We fix $k \in \N$ satisfying $k > \delta^{-100}$ and assume that for $0 \leq n \leq k - 1$ the result holds in the following quantified sense.\medskip

\noindent\textbf{Induction hypothesis.} Let $\gamma \in \fG(\delta_0, M_{\varepsilon})$ and suppose $a \in C^{\infty}(\widehat{\R}^3\setminus \{0\} \times \R \times \R)$ satisfies the symbol condition
\begin{equation}\label{eq: induct 1}
    |\partial_{\xi}^{\alpha}\partial_t^i \partial_s^j a(\xi;t;s)| \leq |\xi|^{-|\alpha|} \quad \textrm{for all $\alpha \in \N_0^3$ and $i$, $j \in \N_0$ with $|\alpha|, i, j \leq M_{\varepsilon}$.}
\end{equation}
For all $0 \leq n \leq k-1$, we have 
\begin{equation*}
    \Big(\int_1^2\|m[a_n](D;t)f\|_{L^6(\R^3)}^6\,\ud t\Big)^{1/6} \leq \bC_{\varepsilon} 2^{-n/6 + \varepsilon n}\|f\|_{L^4(\R^3)}.
\end{equation*}

We remark that if $M_{\varepsilon} \in \N$ is chosen sufficiently large, then all the estimates proved in this paper are uniform over all curves belonging to the class $\fG(\delta_0, M_{\varepsilon})$.

We now turn to the inductive step. Fix $\gamma \in \fG(\delta_0, M_{\varepsilon})$ and $a$ satisfying \eqref{eq: induct 1} and suppose $a_{k, 0}$ is defined as in \eqref{eq: ak0 def}. Provided $\bC_{\varepsilon}$ is chosen sufficiently large in terms of $\delta$, we may apply Lemma~\ref{lem: non-deg loc smoothing} to deduce a favourable bound for the corresponding multiplier $m[a_{k, 0}]$. It remains to show
\begin{equation*}
    \Big(\int_1^2\|m[\fa_k](D;t)f\|_{L^6(\R^3)}^6\,\ud t\Big)^{1/6} \leq (\bC_{\varepsilon}/2) 2^{-k/6 + \varepsilon k}\|f\|_{L^4(\R^3)}
\end{equation*}
for the $\fa_k$ symbols as defined in \eqref{eq: fa symbols}.

For convenience, write 
\begin{equation*}
    U_kf(x,t) := m[\fa_k](D,t)f(x) \quad \textrm{and} \quad U_k^Jf(x,t) := m^J[\fa_k](D,t)f(x) \quad \textrm{for $J \in \fJ(\delta)$.}
\end{equation*}
By fixing an appropriate partition of unity,
\begin{equation*}
    U_kf = \sum_{J \in \fJ(\delta)} U_k^Jf.
\end{equation*}
By an elementary argument (see, for instance,~\cite[Lemma 4.1]{KLO2022}), we have a pointwise bound 
\begin{equation}\label{eq: b-n dec}
    |U_kf(z)| \lesssim \max_{J \in \fJ(\delta)} |U_k^Jf(z)| + \delta^{-1} \sum_{\bJ \in \fJ^{3, \mathrm{sep}}(\delta)} \prod_{J \in \bJ}|U_k^Jf(z)|^{1/3}.
\end{equation}
Taking $L^6$-norms on both sides of \eqref{eq: b-n dec}, we deduce that
\begin{equation}\label{eq: b-n 1}
    \|U_k f\|_{L^6(\R^{3+1})} \lesssim \Big(\sum_{J \in \fJ(\delta)} \|U_k^Jf\|_{L^6(\R^{3+1})}^6 \Big)^{1/6}  + \delta^{-1} \sum_{\bJ \in \fJ^{3, \mathrm{sep}}(\delta)} \Big\|\prod_{J \in \bJ} |U_k^Jf|^{1/3}\Big\|_{L^6(\R^{3+1})}.
\end{equation}

The first term on the right-hand side of \eqref{eq: b-n 1} can be estimated using a combination of rescaling and the induction hypothesis. To this end, let $\tilde{\beta} \in C^{\infty}_c(\widehat{\R}^3)$ be a non-negative function satisfying $\beta \tilde{\beta} = \beta$ and $|\xi| \sim 1$ for $\xi \in \supp \tilde{\beta}$, and define $\tilde{\beta}^k(\xi) := \tilde{\beta}(2^{-k}\xi)$ for $k \in \N_0$. For $J \in \fJ(\delta)$  fix $\tilde{\psi}_J \in C^{\infty}_c(\R)$ satisfying $\supp \tilde{\psi}_J \subseteq 4 \cdot J$, $\tilde{\psi}_J(r) = 1$ for $r \in 2 \cdot J$ and $|\partial_r^N \tilde{\psi}_J(r)| \lesssim_N |J|^{-N}$ for all $N \in \N_0$. We define the Fourier projection $f_J$ of $f$ by 
\begin{equation*}
    \hat{f}_J(\xi) := \chi_J(\xi)  \hat{f}(\xi) \qquad \textrm{where} \qquad\chi_J(\xi) := \tilde{\beta}^k(\xi) \tilde{\psi}_J\circ \theta_2(\xi). 
\end{equation*}
By stationary phase arguments, similar to the proof of Lemma~\ref{J=3 s loc lem}, we then have
\begin{equation*}
  \|U_k^Jf\|_{L^6(\R^{3+1})} \lesssim \|U_k^J f_J\|_{L^6(\R^{3+1})} + 2^{-10k}\|f\|_{L^4(\R^3)} \qquad \textrm{for each $J \in \fJ(\delta)$.}
\end{equation*}
 
Fix $J \in \fJ(\delta)$ with centre $c_J$. By the scaling relation \eqref{eq: mult norm scaled}, we have
\begin{equation*}
    \|U_k^J\|_{L^4(\R^3) \to L^6(\R^{3+1})} \lesssim \delta^{1/2} \|\tilde{U}_k\|_{L^4(\R^3) \to L^6(\R^{3+1})}
\end{equation*}
where $\tilde{U}_k$ is the rescaled operator
\begin{equation*}
     \tilde{U}_kf(x,t) := m_{\tilde{\gamma}}[\tilde{a}](D,t)f(x) \quad \textrm{for} \quad \tilde{\gamma} := \gamma_{c_J, \delta/2}, \quad \tilde{a} := (\fa_k \cdot \psi_J)_{c_J, \delta/2},
\end{equation*}
with the rescalings as defined in \eqref{eq: curve rescale} and \eqref{eq: symbol rescale}. Note that $\tilde{\gamma} \in \fG(\delta_0, M_\varepsilon)$  and, arguing as in the proof of Proposition~\ref{prop:LS J=2}, the symbol $\tilde{a}$ satisfies \eqref{eq: induct 1} (perhaps with a slightly larger constant, but this can be factored out of the symbol). Furthermore, in view of \eqref{eq: resc symb supp}, the symbol $\tilde{a}$ satisfies
\begin{equation*}
    \xisupp \tilde{a} \subset \{ \xi \in \widehat{\R}^3 : |\xi| \sim \delta^3 2^k \}. 
\end{equation*}
In particular, we can write $\tilde{a} = \sum_{n = 0}^{\infty}\tilde{a}_n$ where each $\tilde{a}_n$ is a localised symbol as in \eqref{symbol dec} and the only non-zero terms of this sum correspond to values of $n$ satisfying $2^n \sim \delta^3 2^k$. Thus, by the induction hypothesis, 
\begin{equation*}
     \|\tilde{U}_k\|_{L^4(\R^3) \to L^6(\R^{3+1})} \lesssim \bC_{\varepsilon} (\delta^3 2^k)^{-1/6 + \varepsilon} = \bC_{\varepsilon} \delta^{-1/2 + 3\varepsilon} 2^{-k/6 + \varepsilon k}.
\end{equation*}
Combining these observations,
\begin{align}
    \nonumber
    \Big(\sum_{J \in \fJ(\delta)} \|U_k^Jf\|_{L^6(\R^{3+1})}^6 \Big)^{1/6} &\lesssim \bC_{\varepsilon} \delta^{3\varepsilon} 2^{-k/6 + \varepsilon k} \Big(\sum_{J \in \fJ(\delta)} \|f_J\|_{L^4(\R^3)}^4\Big)^{1/4} \\
    \label{eq: b-n 2}
    &\lesssim \bC_{\varepsilon} \delta^{3\varepsilon} 2^{-k/6 + \varepsilon k} \|f\|_{L^4(\R^3)},
\end{align}
where the final estimate follows from the orthogonality of the $f_J$ via a standard argument.\footnote{Indeed, by interpolation it suffices to show
\begin{equation*}
    \sum_{J \in \fJ(\delta)} \|f_J\|_{L^2(\R^3)}^2 \lesssim \|f\|_{L^2(\R^3)}^2 \quad \textrm{and} \quad \max_{J \in \fJ(\delta)} \|f_J\|_{L^{\infty}(\R^3)} \lesssim \|f\|_{L^{\infty}(\R^3)}.
\end{equation*}
The former follows from Plancherel's theorem and the finite overlap of the Fourier supports of the $f_J$. For the latter, it suffices to show the kernel estimate $\sup_{J \in \fJ(\delta)} \|\mathcal{F}^{-1}\chi_J\|_1 \lesssim 1$. To see this, we apply a rescaling as in the proof of Proposition~\ref{prop:LS J=2}, which transforms $\chi_J$ into a function with favourable derivative bounds.}

On the other hand, each summand in the second term on the right-hand side of \eqref{eq: b-n 1} can be estimated using Proposition~\ref{prop:L4 to L6 trilinear}. In particular, for each fixed $\bJ \in \fJ^{3, \mathrm{sep}}(\delta)$ we have 
\begin{equation}\label{eq: b-n 3}
    \Big\|\prod_{J \in \bJ} |U_k^Jf|^{1/3}\Big\|_{L^6(\R^{3+1})} \lesssim_{\varepsilon} \delta^{-E} 2^{-k/6 + \varepsilon k}\|f\|_{L^4(\R^3)}
\end{equation}
for some absolute constant $E \geq 1$. 

Combining \eqref{eq: b-n 1}, \eqref{eq: b-n 2} and \eqref{eq: b-n 3}, we deduce that
\begin{equation*}
    \|U_k f\|_{L^6(\R^{3+1})} \leq C_{\varepsilon} \big( \bC_{\varepsilon} \delta^{3\varepsilon} + \delta^{-E-4}\big)  2^{-k/6 + \varepsilon k}\|f\|_{L^4(\R^3)},
\end{equation*}
where the constant $C_{\varepsilon} \geq 1$ is an amalgamation of the various implied constants appearing in the preceding argument. Now suppose $\delta > 0$ and $\bC_{\varepsilon}$ have been chosen from the outset so as to satisfy $C_{\varepsilon} \delta^{3\varepsilon} \leq 1/4$ and $\bC_{\varepsilon} \geq  4C_{\varepsilon}\delta^{-E-4}$. It then follows that 
\begin{equation*}
    \|U_k f\|_{L^6(\R^{3+1})} \leq (\bC_{\varepsilon}/2) 2^{-k/6 + \varepsilon k}\|f\|_{L^4(\R^3)},
\end{equation*}
which closes the induction and completes the proof.
\end{proof}




\section{Necessary conditions}\label{sec:nec}

In this section we provide the examples that show that $M_\gamma$ fails to be bounded from $L^p \to L^q$ whenever $(1/p,1/q) \not \in \mathcal{T}$. By a classical result of H\"ormander~\cite{Hormander1960}, $M_{\gamma}$ cannot map $L^p \to L^q$ for any $p > q$. Failure at the point $(1/3,1/3)$ was already shown in \cite{KLO2022} via a modification of the standard Stein-type example for the circular maximal function.  The line joining $(1/3,1/3)$ and $(1/4,1/6)$ is critical via a Knapp-type example, whilst the line joining $(0,0)$ and $(1/4,1/6)$ is critical from the standard example for fixed time averages.

\subsection{The Knapp example}

By an affine rescaling (as in \S\ref{subsec: scaling}), we may assume $\gamma^{(j)}(0) = e_j$ for $1 \leq j \leq 3$, where $e_j$ denotes the standard basis vector. Thus, if $\gamma_{\circ}$ denotes the moment curve as in \S\ref{subsec: scaling}, then $\gamma(s) = \gamma(0) + \gamma_{\circ}(s) + O(s^4)$ for $s \in I$. Furthermore, we may assume without loss of generality that $a := \gamma_3(0) > 0$. Given $\delta > 0$, let $f_{\delta} := \bbone_{R_{\delta}}$ where
\begin{equation*}
    R_{\delta} := \{ y \in \R^3 : |y_j| < \delta^j, \, 1 \leq j \leq 3\}.
\end{equation*}
Clearly, $\| f_\delta \|_{L^p(\R^3)} \lesssim \delta^{6/p}$. Consider the domain
\begin{equation*}
    E_{\delta} := \big\{ x \in \R^3:  |x_j - x_3\gamma_j(0)/a| < \delta^j/2, \, \textrm{ $j=1,2$}, \,\,\,\, a \leq x_3 \leq 2a \big\}.
\end{equation*}
By the moment curve approximation, there exists a constant $c_{\gamma} > 0$ such that if $|s| < c_{\gamma}\delta$, then the following holds. If $x \in E_{\delta}$ and $t(x) := x_3/a$, then
\begin{equation*}
    |x_j - t(x)\gamma_j(s)| \leq |x_j - t(x)\gamma_j(0)| +|t(x)| |\gamma_j(0) - \gamma_j(s)| < \delta^j \quad  \textrm{for $j=1,2$}
\end{equation*}
and
\begin{equation*}
   |x_3 - t(x)\gamma_3(s)| = |t(x)||\gamma_3(0) - \gamma_3(s)| < \delta^3.
\end{equation*}
Thus, we conclude that $x - t(x) \gamma(s) \in R_{\delta}$ for all $|s| < c_{\gamma}\delta$ and therefore
\begin{equation*}
   \| M_\gamma f_\delta \|_{L^q(\R^3)} \gtrsim \delta |E_\delta|^{1/q} \gtrsim \delta^{1+3/q}.
\end{equation*}
The bound $\| M_\gamma f_\delta \|_{L^q(\R^3)} \lesssim \| f_\delta \|_{L^p(\R^3)}$ therefore implies $\delta^{1+3/q} \lesssim \delta^{6/p}$; letting $\delta \to 0$, this can only hold if $1+\frac{3}{q} \geq \frac{6}{p}$. This gives rise to the line joining $(1/3,1/3)$ and $(1/4,1/6)$ in Figure~\ref{fig: Riesz}.

\subsection{Dimensional constraint} This is the standard example for $L^p \to L^q$ boundedness for the fixed time averages. Given $0< \delta < 1$, consider $g_\delta=\bbone_{N_\delta(\gamma)}$ where
\begin{equation*}
    N_\delta(\gamma):=\{x \in \R^3 : |x+ \gamma(s) | \leq \delta \,\, \text{ for some $s \in I$}\}.
\end{equation*}
Clearly, $\| g_\delta \|_{L^p(\R^3)} \lesssim \delta^{2/p}$. Furthermore, $x - \gamma(s) \in N_\delta(\gamma)$ for all $|x| \leq \delta$. 
This readily implies $M_\gamma g_\delta(x) \geq A_1 g_\delta(x) \gtrsim 1$ for $|x|\leq \delta$, and consequently, $\| M_\gamma g_\delta \|_{L^q(\R^3)} \gtrsim \delta^{3/q}$.  The bound $\| M_\gamma g_\delta \|_{L^q(\R^3)} \lesssim \| f_\delta \|_{L^p(\R^3)}$ implies $\delta^{3/q} \lesssim \delta^{2/p}$; letting $\delta \to 0$, this can only hold if $\frac{3}{q}\geq \frac{2}{p}$. This gives rise to the line joining $(0,0)$ and $(1/4, 1/6)$ in Figure~\ref{fig: Riesz}.




\appendix

\section{Localised multilinear restriction estimates}\label{sec: appendix}

Here we present the proof of Theorem~\ref{thm: local BCT}. We use a simple Fubini argument to essentially reduce the problem to particular cases of the multilinear restriction inequalities from~\cite[Theorem 1.3]{BBFL2018} and~\cite[Theorem 5.1]{BCT2006}. More precisely, we require low-regularity versions of these results which apply to $C^{1,1/2}$-hypersurfaces. However, the arguments of~\cite{BBFL2018} and~\cite{BCT2006} extend to cover the $C^{1,\alpha}$-class for any $\alpha > 0$ by incorporating minor modifications to the induction-on-scale scheme as in the proof of~\cite[Theorem 3.6]{KLO2022}; we omit the details. 

\begin{proof}[Proof (of Theorem~\ref{thm: local BCT})] Let $\delta > 0$ be a small constant, which is independent of $\mu$ and $R$ and chosen to satisfy the forthcoming requirements of the proof. We may assume without loss of generality that $0 < \mu < \delta$, since otherwise the desired estimate follows from the $C^{1, 1/2}$ extension of the Bennett--Carbery--Tao multilinear inequality~\cite[Theorem 5.1]{BCT2006}.

By localising the operators and applying a suitable rotation to the coordinate domain, we may assume that there exists an open domain $U_3' \subseteq \widehat{\R}^2$ and a smooth map $\gamma \colon U_3' \to \R$ such that
\begin{equation*}
     \big\{\xi \in \supp a_3 : u(\xi) = 0 \big\} = \{(s, \gamma(s)) : s \in U_3' \}
\end{equation*}
and, moreover, 
\begin{equation*}
    \supp a_3 \subseteq \big\{(s, \gamma(s) + r) : s \in U_3' \textrm{ and } |r| < \mu \big\}.
\end{equation*}
By differentiating the defining identity for $\gamma$, we observe that 
\begin{equation}\label{eq: loc BCT 1}
    (\partial_{\xi_j}u)(s, \gamma(s)) + (\partial_{s_j}\gamma)(s) (\partial_{\xi_3}u)(s, \gamma(s)) = 0 \qquad \textrm{for $j = 1, 2$.}
\end{equation}

By a change of variables, we may write
\begin{equation*}
    E_3f_3(x,t) = \int_{- \mu}^\mu e^{ir x_3} E_{3,r} f_{3,r}(x,t) \,\ud r
\end{equation*}
where $f_{3,r}(s) := f_3(s, \gamma(s) + r)$ and
\begin{equation*}
    E_{3,r}g(x,t) := \int_{\widehat{\R}^2} e^{i(\inn{\Gamma(s)}{x} + t Q_3(s, \gamma(s) + r))} a_{3,r}(s)g(s)\,\ud s 
\end{equation*}
for $\Gamma(s) := (s, \gamma(s))$ and $a_{3,r}(s) := a_3(s, \gamma(s) + r)$. For each fixed $|r| < \mu$, the operator $E_{3,r}$ is the extension operator associated to the $2$-surface
\begin{equation*}
    \Sigma_{3,r}' := \big\{(s, \gamma(s), Q_3(s, \gamma(s) + r)) : s \in U_3' \big\}.
\end{equation*}
When $r = 0$, it follows from \eqref{eq: loc BCT 1} that 
\begin{equation}\label{eq: loc BCT 2}
 \mathrm{span} \Big\{   \begin{pmatrix}
        - \nabla Q_3\circ\Gamma(s) \\
        1
    \end{pmatrix}, \,
    \begin{pmatrix}
        \nabla u \circ \Gamma(s) \\
        0
    \end{pmatrix} \Big\} = N_{\xi} \Sigma_{3,0}' \quad \textrm{for} \quad \xi := (\Gamma(s) , Q_3\circ \Gamma(s));
\end{equation}
that is, the span of the two vectors is equal to the normal space to $\Sigma_{3,0}'$ at $\xi$.

After applying a simple Fubini--Tonelli argument, the problem is reduced to showing
    \begin{equation}\label{eq: loc BCT 3}
   \int_{-\mu}^{\mu} \int_{B(0,R)}\prod_{j=1}^2|E_jf_j(x,t)||E_{3,r}f_{3,r}(x,t)|\,\ud x \ud t \ud r \lesssim_{\mathbf{Q}, \varepsilon} R^{\varepsilon} \mu^{1/2} \prod_{j=1}^3 \|f_j\|_{L^2(U_j)}.
\end{equation}
The key claim is that for each $|r| < \mu$, the trio of extension operators $(E_1, E_2, E_{3,r})$ satisfy the hypothesis of~\cite[Theorem 1.3]{BBFL2018}. In particular, provided $\delta > 0$ is chosen sufficiently small, our transversality hypothesis \eqref{eq: trans} implies that the normal spaces to the submanifolds $\Sigma_1,\Sigma_2,\Sigma_{3,r}$ factorise the space $\widehat{\R}^4$ in the sense that
\begin{equation*}
    N_{\xi_1}\Sigma_1 \oplus N_{\xi_2}\Sigma_2\oplus N_{\xi_3}\Sigma_{3,r}' = \widehat{\R}^4
\end{equation*}
for all $|r|<\mu$ and all choices of $\xi_1\in\Sigma_1,\xi_2\in\Sigma_2,\xi_3\in\Sigma_{3,r}'$. To see this, we first prove the $r = 0$ case by combining \eqref{eq: loc BCT 2} and \eqref{eq: trans}, and then extend to all $|r| < \mu < \delta$ using continuity. Consequently, we can use the formula proved in~\cite[Proposition 1.2]{bennett2010some} together with \eqref{eq: trans} to conclude that the Brascamp--Lieb constant associated to the orthogonal projections onto the tangent spaces $T_{\xi_1}\Sigma_1$, $T_{\xi_2}\Sigma_2$, $T_{\xi_3}\Sigma_{3,r}$ (with Lebesgue exponents $p_1 = p_2 = p_3 = 1/2$) is uniformly bounded. We refer to~\cite{bennett2010some, BBFL2018} for the relevant definitions. This is precisely the hypothesis of~\cite[Theorem 1.3]{BBFL2018} and invoking (a suitable $C^{1,1/2}$ generalisation of) this result we obtain 
\begin{equation*}
    \int_{B(0,R)}\prod_{j=1}^2|E_jf_j(x,t)||E_{3,r}f_{3,r}(x,t)|\,\ud x \ud t \lesssim_{\mathbf{Q}, \varepsilon} R^{\varepsilon} \prod_{j=1}^2 \|f_j\|_{L^2(U_j)} \|f_{3,r}\|_{L^2(U_3')}
\end{equation*}
uniformly over all $|r| < \mu$.  We integrate both sides of this inequality with respect to $r$ and apply the Cauchy--Schwarz inequality to deduce that
\begin{align*}
     \int_{-\mu}^{\mu} \int_{B(0,R)}\prod_{j=1}^2|E_jf_j(x,t)|&|E_{3,r}f_{3,r}(x,t)|\,\ud x \ud t \ud r \\
     &\lesssim_{\mathbf{Q}, \varepsilon} R^{\varepsilon} \mu^{1/2} \prod_{j=1}^2 \|f_j\|_{L^2(U_j)} \Big( \int_{-\mu}^{\mu}\|f_{3,r}\|_{L^2(U_3')}^2\,\ud r \Big)^{1/2}.
\end{align*}
The desired estimate \eqref{eq: loc BCT 3} now follows by reversing the original change of variables. 
\end{proof}




\bibliography{Reference}
\bibliographystyle{amsplain}

\end{document}